\documentclass[12pt]{amsart}
\usepackage{amsmath,amsthm,amsfonts,amssymb,times,mathtools}
\usepackage{latexsym,epsfig,subfigure,afterpage} 

\usepackage{xcolor}
\usepackage[colorlinks=true, allcolors=darkblue, urlcolor=darkblue, linktocpage=true, linkcolor=darkblue, citecolor=purple]{hyperref}

\DeclareMathAlphabet{\curly}{U}{rsfs}{m}{n}

\theoremstyle{remark}
\newtheorem{example}{Example}
\newtheorem{remark}{Remark}

\theoremstyle{plain}
\newtheorem{defn}{Definition}

\newtheorem{lemma}{Lemma}[section]
\newtheorem{thm}{Theorem}

\numberwithin{equation}{section}

\newcommand{\bal}{\[\begin{aligned}}
\newcommand{\eal}{\end{aligned}\]}

\newcommand{\be}{\begin{equation}}
\newcommand{\ee}{\end{equation}}

\newcommand{\benn}{\begin{equation*}}
\newcommand{\eenn}{\end{equation*}}

\renewcommand{\leq}{\leqslant}

\renewcommand{\geq}{\geqslant}

\DeclareMathOperator{\sgn}{sgn}

\newcommand\ba{\begin{eqnarray}}
\newcommand\ea{\end{eqnarray}}

\newcommand{\lb}{\left(}
\newcommand{\rb}{\right)}
\newcommand{\lcb}{\left\{}
\newcommand{\rcb}{\right\}}
\newcommand{\lsb}{\left[}
\newcommand{\rsb}{\right]}

\newcommand{\mrm}{\mathrm}
\newcommand{\nn}{\nonumber}

\newcommand{\ceil}[1]{\left\lceil #1 \right\rceil }

\colorlet{darkblue}{blue!70!black}
\colorlet{darkgreen}{green!70!black}
\colorlet{darkorange}{orange!70!black}
\colorlet{darkcyan}{cyan!70!black}

\def\bs#1{{\color{red}#1}}
\def\extratext#1{#1}

\usepackage{graphicx}
\begin{document}
\title[A visual perspective on the BSD conjecture]{A visual perspective on the Birch and Swinnerton-Dyer conjecture through a family of approximations of $L$-functions}
\author{Maria Nastasescu, Bogdan Stoica, and Alexandru Zaharescu}

\address{Maria Nastasescu: Department of Mathematics, 2033 Sheridan Road, Northwestern University, Evanston, IL 60207, USA}
\email{mnastase@math.northwestern.edu}

\address{Bogdan Stoica: Martin A. Fisher School of Physics, Brandeis University, Waltham, MA 02453, USA}
\email{bstoica@brandeis.edu}

\address{Alexandru Zaharescu: Department of Mathematics, 1409 West Green Street, University
of Illinois at Urbana-Champaign, Urbana, IL 61801, USA, 
and Simion Stoilow Institute of Mathematics of the Romanian Academy, 
P.O. Box 1-764, RO-014700 Bucharest, Romania}
\email{zaharesc@illinois.edu}

\begin{abstract}

We investigate the properties of a family of approximations of the Hasse-Weil $L$-function associated to an elliptic curve $E$ over $\mathbb{Q}$. We give a precise expression for the error of the approximations, and provide a visual interpretation of the analytic rank $m$ of $E$ as a sequence of near regular polygons around the center of the critical strip, each with vertices at the zeros of the approximations.

\end{abstract}


\maketitle

\tableofcontents

\section{Introduction}

\noindent In this paper we define a family of approximations for the Hasse-Weil $L$-function associated to an elliptic curve over $\mathbb{Q}$. Our construction generalizes the prescription introduced by Matiyasevich \cite{Ma1}, who defined a family of approximations of the Riemann zeta function by considering certain regularized truncated Euler products.  We prove that our family approximates with high precision the corresponding $L$-function, and we give a precise expression for the error term.

In the case of the Riemann zeta function, the approximations introduced in \cite{Ma1} are conjectured to satisfy a Bounded Riemann Hypothesis \cite{Ma1,NastasescuZaharescu}. This means that for any integer $k\geq1$, there exists a level of the approximation (i.e. a number of primes included in the approximation) such that the first $k$ zeros of the approximation, counted with possible multiplicity in order of ascending positive imaginary part, are on the critical line. Matiyasevich also considered approximations to the $L$-function associated to the Ramanujan tau function. These too are expected to satisfy a bounded~RH.

In contrast, for $L$-functions associated to elliptic curves, we will show that the Bounded Riemann Hypothesis does not hold for our approximations, see for example Figures~\ref{fig1} -- \ref{fig4} (pages 5, 6). We instead are able to provide a visual interpretation of the analytic rank $m$ of the elliptic curves in terms of the zeros of the approximations closest to the center of the critical strip. As we increase the order of the approximation, these zeros arrange themselves as the vertices of increasingly smaller near-regular polygons around the center.

Moreover, we find that the zeros of these approximations encode important arithmetic information pertaining to the elliptic curve. In particular, counting the number of zeros in a small disk around the center, we can visualize the equality between the analytic rank and the arithmetic rank for a given elliptic curve, providing a visual perspective of the Birch and Swinnerton-Dyer conjecture. Additionally, these approximations recover information about the leading coefficient in the Taylor series expansion of the $L$-function at the central point (see Theorem \ref{thm3ngon}), which in turn captures important information on the elliptic curve (see the presentation of the BSD conjecture \cite{wilescitation}). The behavior of these polygons is also consistent with the Sato-Tate conjecture, since the size of the polygons is directly influenced by the size of the corresponding coefficients $a_{p_{N+1}}$. 

In Figures \ref{fig1}, \ref{fig2} the zeros of each approximation near the central point appear to form the vertices of a near square, and in Figure \ref{fig3} the zeros appear to lie close to the vertices of a regular hexagon around the central point. In Figure~\ref{fig4} four zeros are again close to the vertices of a square, and a fifth zero is at the central point. Figures \ref{fig1}--\ref{fig2}, \ref{fig3}, \ref{fig4} correspond to $L$-functions of elliptic curves of rank 4, 6, and 5 respectively. We will study this phenomenon in Theorem~\ref{thm3ngon} below. 

We now explain our construction. This construction takes as input the local factors, and returns an approximation $\Lambda_N(E,s)$ for the completed $L$-function $\Lambda(E,s)$ associated to the local factors. It is a requirement for the construction that the local factors correspond to an $L$-function, that is $\Lambda(E,s)=g(s)\prod_{p=2}^\infty L_p(E,s)$, for $s$ in a right half-plane. Here $\Lambda(E,s)$ is the completed $L$-function, $g(s)$ is the local factor at infinity, and $L_p(E,s)$ is the local factor at place $p$.

The first step is to construct a finite Euler product of the local factors $L_p(E,s)$, multiplied over all places (Archimedean and finite) up to a largest prime $p_N$, i.e.
\be
\label{eq11Euler}
\Lambda_N^\mrm{Euler}(E,s) \coloneqq g(s) \prod_{p=2}^{p_N} L_p(E,s).
\ee 
When approximating the Riemann zeta function we have $L_p(s)=1/(1-p^{-s})$. \extratext{In the case of $L$-functions associated to elliptic curves, if $E$ has good reduction at $p$ we have}
\be
L^{-1}_p(E,s) = 1 - a_p p^{-(s+1/2)} + p^{-2s},
\ee
and if $E$ has bad reduction at $p$ we have
\be
L^{-1}_p(E,s) =1 + a_p p^{-s-\frac{1}{2}},
\ee
with $a_p\in\{0,\pm 1\}$.

\extratext{Matiyasevich's construction relies on extracting the holomorphic part of $\Lambda_N^\mrm{Euler}(s)$. Unlike the function $\Lambda(E,s)$ that we want to approximate, the finite Euler product \eqref{eq11Euler} has an infinite number of poles in the complex $s$-plane. We will need to remove these poles if we want to recover $\Lambda(E,s)$ from the finite Euler product $\Lambda_N^\mrm{Euler}(E,s)$. In order to do this, we define the \emph{principal part} $\Lambda_N^\mrm{pp}(E,s)$ of $\Lambda_N^\mrm{Euler}(E,s)$ as the sum of principal parts of the Laurent expansions at all these poles, and we subtract this principal part from $\Lambda_N^\mrm{Euler}(E,s)$, in order to obtain the holomorphic part of $\Lambda_N^\mrm{Euler}(E,s)$. The sum over poles in $\Lambda_N^\mrm{pp}(E,s)$ quickly converges, from the properties of the Gamma function. Unlike $\Lambda(E,s)$, the functions $\Lambda_N^\mrm{Euler}(E,s)$ and $\Lambda_N^\mrm{pp}(E,s)$ do not satisfy a functional equation, so one must (anti-)symmetrize under $s\to1-s$ to obtain the approximation $\Lambda_N(E,s)$ for the $L$-function. In precise terms, we define}
\be
\label{eq11}
\Lambda_N(E,s) \coloneqq \Lambda_N^\mrm{Euler}(E,s) - \Lambda_N^\mrm{pp}(E,s) + (-1)^P \lsb \Lambda_N^\mrm{Euler}(E,1-s) - \Lambda_N^\mrm{pp}(E,1-s) \rsb .
\ee
In the case of the Riemann zeta function $(-1)^P\coloneqq 1$ and the equivalent construction to Eq. \eqref{eq11} approximates the completed zeta function $\xi(s)$ to order $\mathcal{O}\lb \exp( -\mathcal{K} p_{N+1}^2 )\rb$, with \extratext{$p_{N+1}$ the next prime number following $p_N$, and} $\mathcal{K}$ a numerical constant that does not depend on~$N$ (see \cite{NastasescuZaharescu}). In the case of $L$-functions associated to elliptic curves $(-1)^P$ is the root number, and we will show that Eq. \eqref{eq11} approximates $\Lambda(E,s)$ to order $\mathcal{O}\lb \exp( -\mathcal{K'} p_{N+1} )\rb$. More precisely, we have the following theorem.

\begin{thm}
\label{theoremintro}
Let $E$ be an elliptic curve $E$ over $\mathbb{Q}$ \extratext{of conductor $C$ and root number $(-1)^P$, and let $\Lambda(E,s)$ be its completed $L$-function.} For any $\mathcal{R}>0$, there exist constants $B_1(\Lambda,\mathcal{R})>0$, $B_2(\Lambda,\mathcal{R})>0$ depending only on $\Lambda(E,s)$ and $\mathcal{R}$, such that for any $N\geq 1$, we have
\ba
\label{eq13intro}
& &\left|\Lambda(E,s) - \Lambda_N(E,s) -\frac{ a_{p_{N+1}} C^\frac{1}{4}}{\pi p_{N+1}} e^{-\frac{2 \pi  p_{N+1}}{\sqrt{C}}}  \lsb 1 + (-1)^P + \frac{(2s-1)C^\frac{1}{2}}{4\pi p_{N+1}} \lb 1 - (-1)^P \rb \rsb \right| \\
&\leq & B_1(\Lambda,\mathcal{R}) \frac{ \left| a_{p_{N+1}}\right| }{p^3_{N+1}} e^{-\frac{2 \pi  p_{N+1}}{\sqrt{C}}} + B_2(\Lambda,\mathcal{R}) \frac{1}{\lb p_{N+2}\rb^{1-(-1)^P/2}} e^{-\frac{2\pi p_{N+2}}{\sqrt{C}}}, \nn
\ea
uniformly for all $s\in\mathbb{C}$ with $|s|\leq \mathcal{R}$, where $\Lambda_N(E,s)$ is defined in \extratext{Eqs. \eqref{eq11Euler}, \eqref{eq11}}.
\end{thm}

\extratext{This theorem shows that the functions $\Lambda_N(E,s)$ are indeed approximations for the $L$-function $\Lambda(E,s)$, and it gives us a sharp bound for the difference $\Lambda_N(E,s)-\Lambda(E,s)$. A fact that will become important later is that for almost all $N$, Eq. \eqref{eq13intro} provides an asymptotic formula for the difference $\Lambda_N(E,s)-\Lambda(E,s)$. More precisely, when the coefficient $a_{p_{N+1}}$ is not too small and when the difference between the consecutive primes $p_{N+2}-p_{N+1}$ is not too small, the two terms on the right-hand side of Eq. \eqref{eq13intro} are dominated by the term involving $a_{p_{N+1}}$ on the left-hand side.}

Theorem \ref{theoremintro} implies the following presentation for the completed $L$-function, in analytic convention.
\begin{thm}
\label{thmextra}
Let $E$ be an elliptic curve $E$ over $\mathbb{Q}$ and let $\Lambda(E,s)$ be its completed $L$-function. For any $s_0\in \mathbb{C}$ and $\sigma\in\mathbb{R}$ such that $\sigma > \Re(s_0)$, $\sigma > 1- \Re(s_0)$, we have
\be
\Lambda(E,s_0) = \frac{1}{2\pi i} \int_{\Re(s)=\sigma} g(s) \sum_{n = 1}^\infty \frac{ a_n }{n^{s+\frac{1}{2}}} \lb \frac{1}{s-s_0} + (-1)^P \frac{1}{s-1+s_0} \rb ds,
\ee
where integers $a_n$ are the coefficients of the Dirichlet series.
\end{thm}

\extratext{Using the asymptotic formula from Eq. \eqref{eq13intro}, one can prove a theorem that explains the phenomenon shown in Figures~\ref{fig1}~--~\ref{fig4}. Given an elliptic curve $E$, its analytic rank $m$ is visually manifested in the geometric configuration of the nearest zeros to the central point of the family of approximations $\Lambda_N (E,s)$. More precisely, for a family $\{N_k\}_{k \in I}$ of sufficiently large integers, the zeros of $\Lambda_{E,N_k}(s)$ near $s=1/2$ arrange themselves close to the vertices of regular $m$-gons (if $m$ is even) and as the vertices of a $(m-1)$-gon, plus a zero at $s=1/2$, if $m$ is odd. As $k \to \infty$, the zeros converge to the central point. We make these statements precise below.}

\extratext{Suppose the order of vanishing of $\Lambda(E,s)$ at $s=1/2$ is $m$. From the functional equation $\Lambda(E,s)=(-1)^P\Lambda(E,1-s)$ we have that $(-1)^m=(-1)^P$, and the $L$-function is even or odd under $s\to1-s$, so that we have the Taylor series expansion around $s=1/2$,
\be
\label{eq15TaylorL}
\Lambda(E,s) = c_m \lb s - 1/2\rb^m + c_{m+2} \lb s - 1/2\rb^{m+2} + \dots.
\ee}

\extratext{\begin{defn}[Limit set configurations]
\label{deflimsets}
For even positive integer $m \geq 2$, let
\be
\mathcal{S}^\mrm{even,+}_m \coloneqq \left\{ \left| \frac{2C^\frac{1}{4}}{\pi c_m}\right|^\frac{1}{m} e^{2\pi i j /m}:j \in \mathbb{Z}, 1\leq j \leq m \right\}
\ee
and
\be
\mathcal{S}^\mrm{even,-}_m\coloneqq \left\{ \left| \frac{2C^\frac{1}{4}}{\pi c_m}\right|^\frac{1}{m} e^{\pi i (2j-1) /m}:j \in \mathbb{Z}, 1\leq j \leq m\right\}.
\ee For odd positive integer $m > 2$, let $\mathcal{S}^\mrm{odd,+}_m\coloneqq \{0\}\cup \mathcal{S}_{m-1}^\mrm{even,+}$ and $\mathcal{S}^\mrm{odd,-}_m\coloneqq \{0\}\cup \mathcal{S}_{m-1}^\mrm{odd,-}$.
\end{defn}}

\extratext{For $m>3$ the sets in Definition \ref{deflimsets} describe regular $m$-gons in the complex plane, with a point at the center when $m$ is odd.}

\extratext{\begin{defn}
Let $\{\rho^{(N)}_1,\dots \rho^{(N)}_m\}$ be the set of $m$ closest zeros of $\Lambda_N(E,s)$ to the central point~$1/2$, and define
\be
A_{\Lambda,N} \coloneqq \lcb \left| \frac{ p_{N+1}}{a_{p_{N+1}}}e^\frac{2\pi p_{N+1}}{\sqrt{C}} \right|^{\frac{1}{m}}\lb\rho^{(N)}_j - \frac{1}{2}\rb : 1\leq j\leq m \rcb.
\ee
\end{defn}}

\extratext{\begin{thm}
\label{thm3ngon}
Let $\Lambda(E,s)$ be an $L$-function associated to an elliptic curve E with conductor $C$, with multiplicative reduction at at least one place. Assume that the order of vanishing at the central point $s=1/2$ is $m\geq 2$, and let $c_m$ be the first nonzero coefficient in the Taylor series at $s=1/2$. Then there exists a set of primes $\mathfrak{B}_\Lambda$ of full density in the set of all primes, that decomposes as
\be
\mathfrak{B}_\Lambda = \mathfrak{B}^{+}_\Lambda \sqcup \mathfrak{B}^{-}_\Lambda,
\ee
with $\mathfrak{B}_\Lambda^+$, $\mathfrak{B}_\Lambda^-$ having density $1/2$ each, such that:
\begin{enumerate}
\item The sequence of sets $A_{\Lambda,N}$ converges to $\mathcal{S}^+_m$ when $p_N\in \mathfrak{B}^+_\Lambda$ tends to infinity,
\item The sequence of sets $A_{\Lambda,N}$ converges to $\mathcal{S}^-_m$ when $p_N\in \mathfrak{B}^-_\Lambda$ tends to infinity.
\end{enumerate}
\end{thm}}

\extratext{When the curve has additive reduction at all the bad places, Theorem \ref{thm3ngon} holds conditional on the Sato-Tate conjecture.}

\extratext{In Section \ref{sec2prelim} we review notation and introduce some useful lemmas. In Section \ref{sec3Matiya} we present in detail Matiyasevich's construction, adapted to the case of $L$-functions associated to elliptic curves over~$\mathbb{Q}$. In Section \ref{sec4someproofs} we give the proof of Theorem \ref{theoremintro} that our construction indeed produces approximations for the given $L$-functions. Finally, in Section \ref{sec5proofsforpolygons} we give the proof of Theorem \ref{thm3ngon}.}

\begin{figure}[htp]
\centering
\includegraphics[scale=0.7]{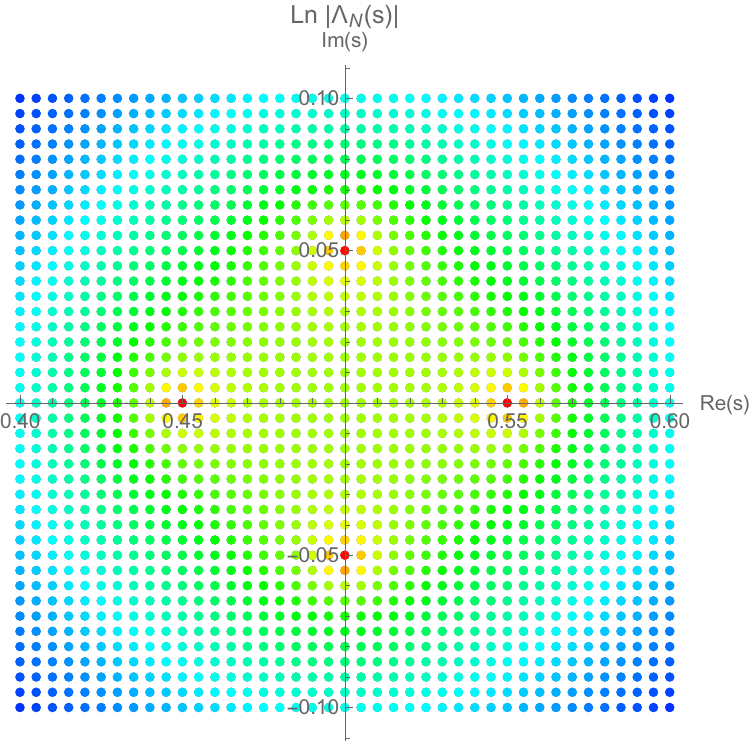}
\includegraphics[scale=0.7]{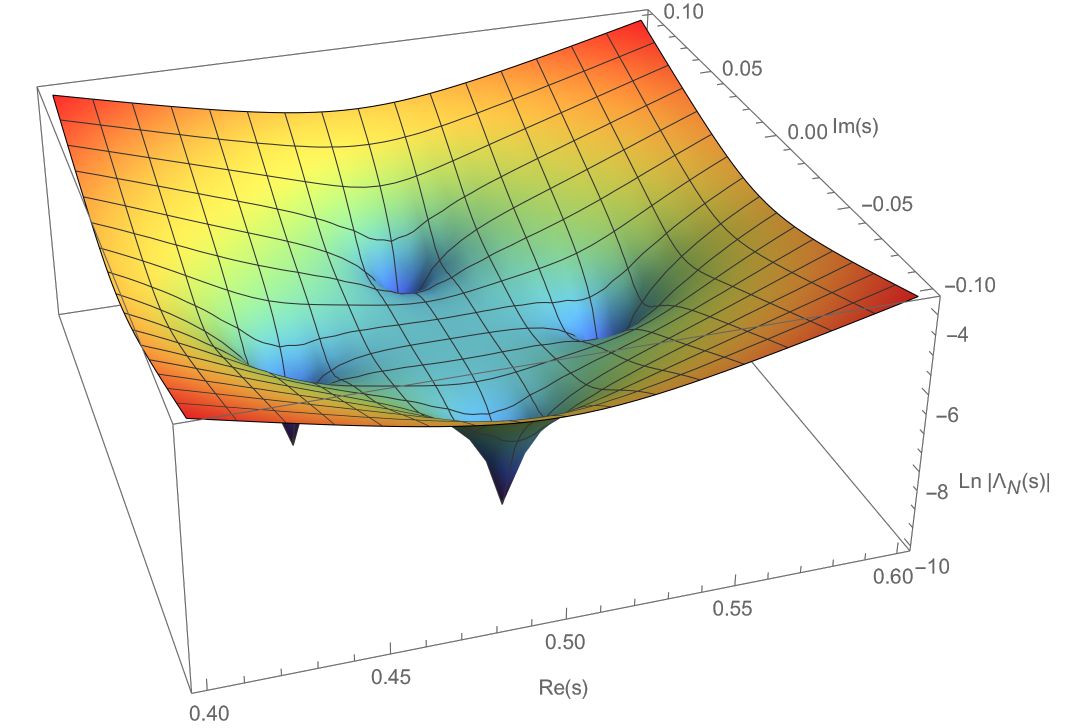}
\caption{The closest 4 zeros of the approximation $\Lambda_N(E,s)$ to the central point, for all primes up to (and including) $p_{N=78}=397$. This approximation is for the $L$-function associated to elliptic curve $y^2+xy=x^3-x^2-79x+289$ of conductor $C=234446$, LMFDB label 234446.a1. The zeros form an approximate square with vertices on the real axis and critical line.}
\label{fig1}
\end{figure}

\begin{figure}[htp]
\centering
\includegraphics[scale=0.7]{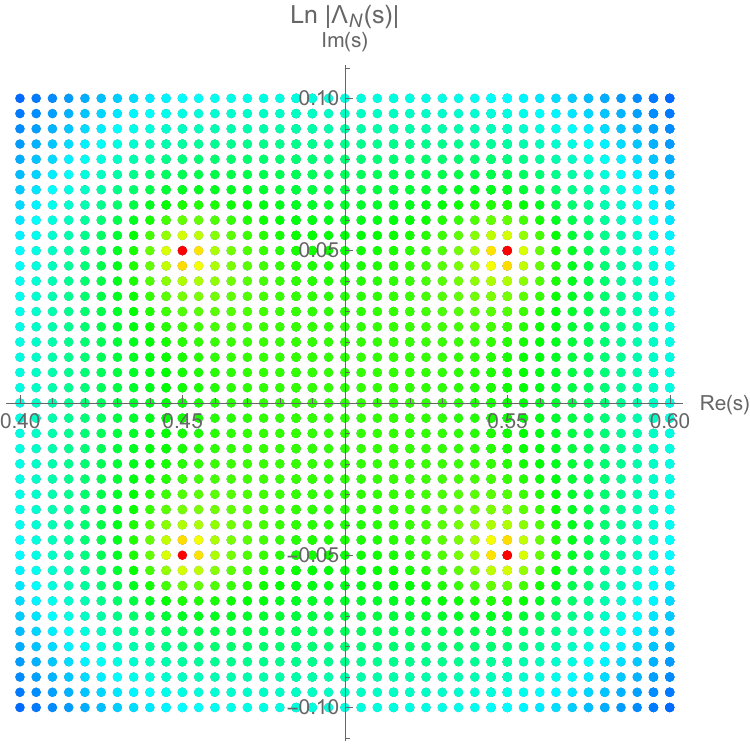}
\includegraphics[scale=0.7]{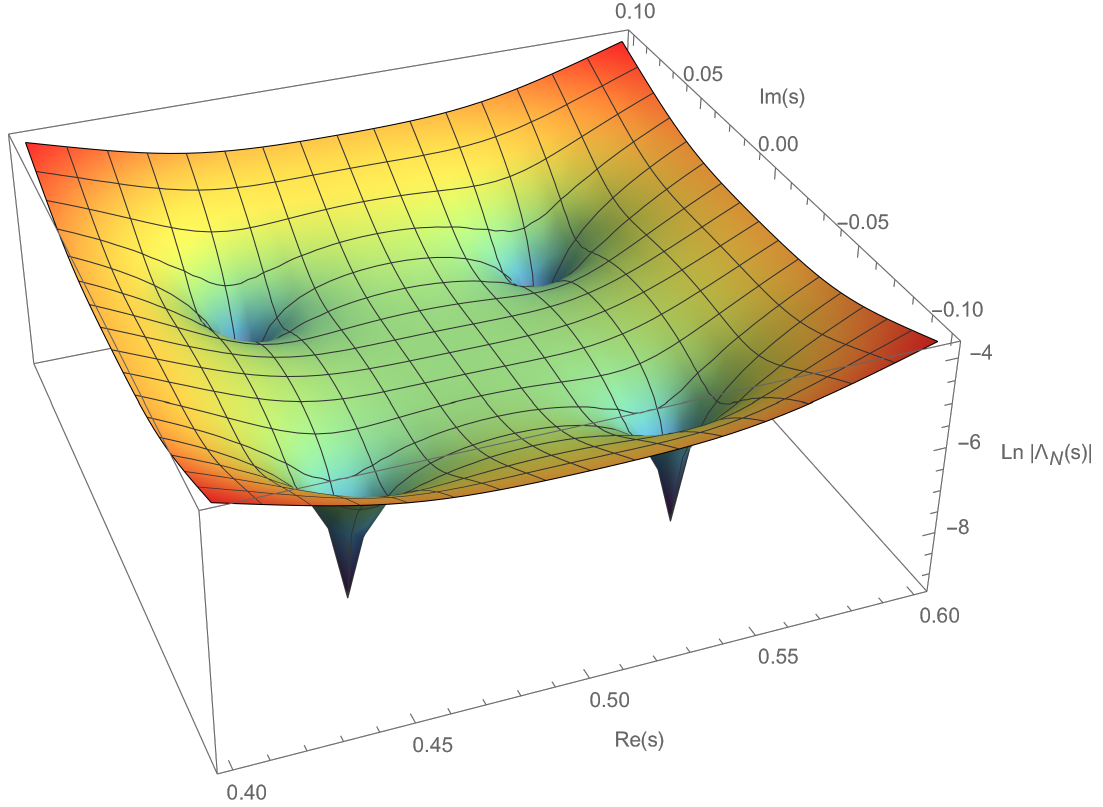}
\caption{The closest 4 zeros of the approximation $\Lambda_N(E,s)$ to the central point, for all primes up to (and including) $p_{N=79}=401$. This approximation is for the $L$-function associated to elliptic curve $y^2+xy=x^3-x^2-79x+289$ of conductor $C=234446$, LMFDB label 234446.a1. In this case the zeros form an approximate square with vertices on diagonals, signifying that $\Lambda_{N=78}(1/2)$ and $\Lambda_{N=79}(1/2)$ have opposite~signs.}
\label{fig2}
\end{figure}

\begin{figure}[htp]
\centering
\includegraphics[scale=0.7]{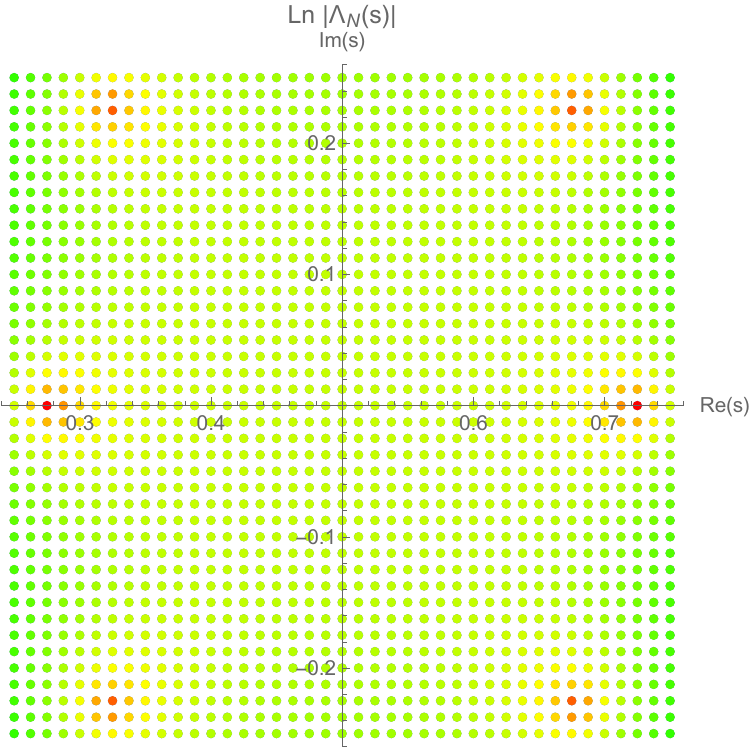}
\includegraphics[scale=0.7]{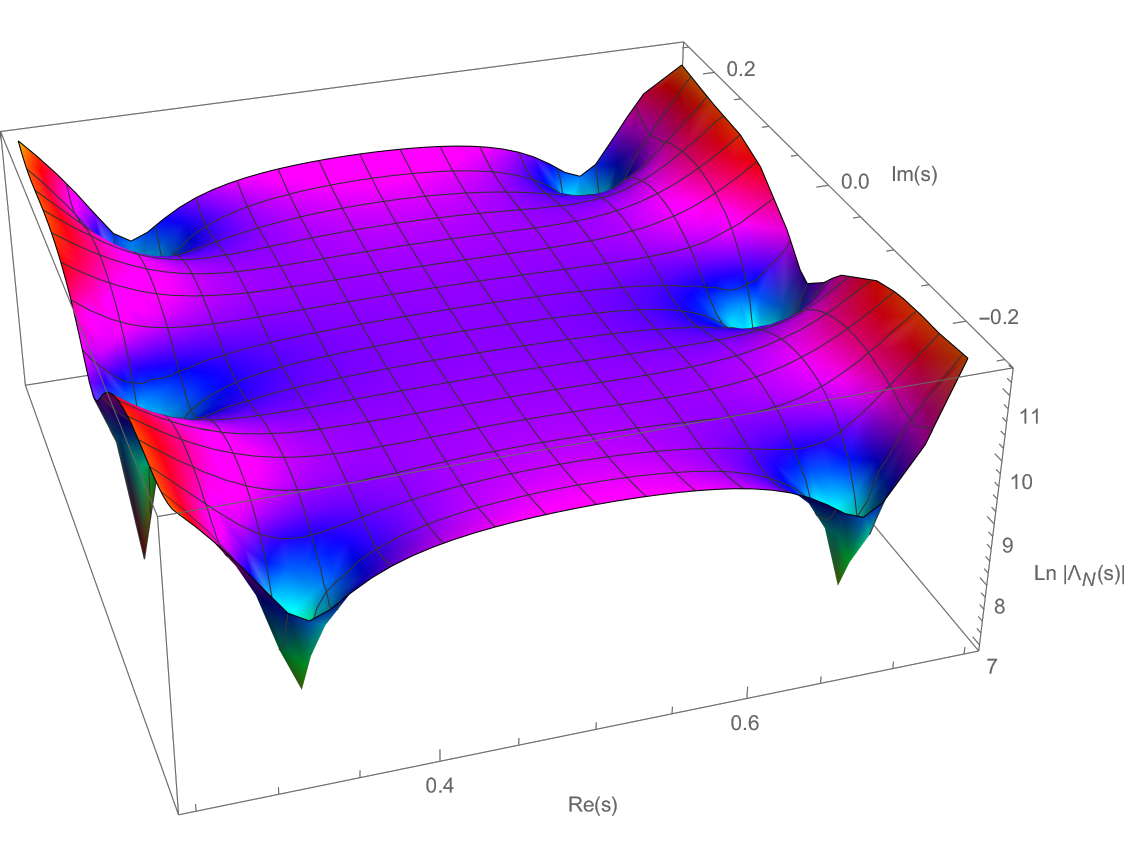}
\caption{The closest 6 zeros of the approximation $\Lambda_N(E,s)$ to the central point, for all primes up to (and including) $p_{N=320}=2129$. This approximation is for the $L$-function associated to elliptic curve $y^2-x^3-5858 x^2+111546435 x=0$ of conductor $C=26799137200956120$ \cite{PenneyPomerance}. }
\label{fig3}
\end{figure}

\begin{figure}[htp]
\centering
\includegraphics[scale=0.7]{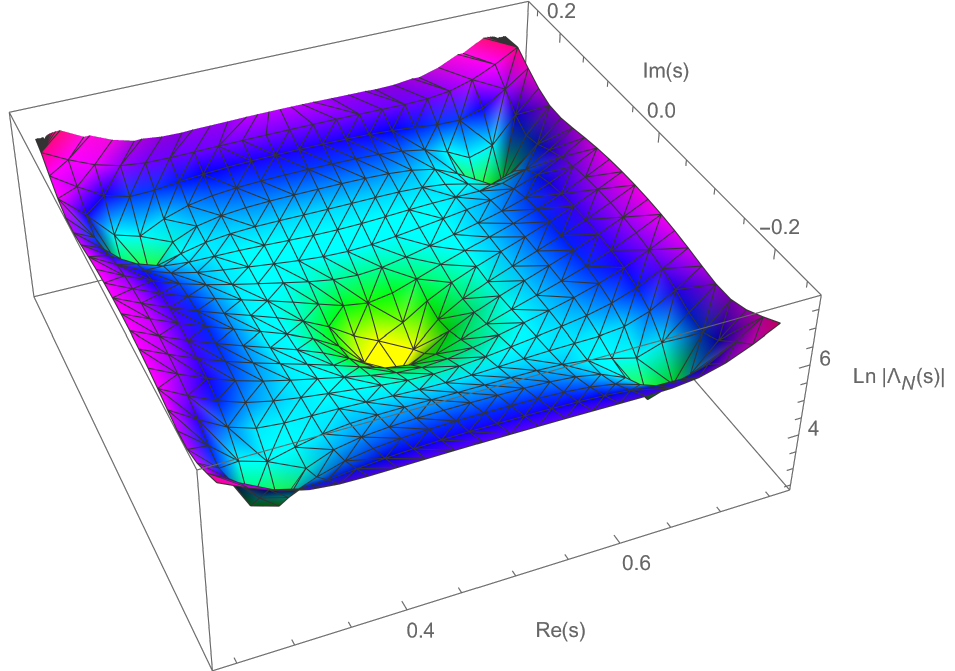}
\caption{Plot of $\ln |\Lambda_N(E,s)|$ in the central region for all primes up to (and including) $p_{N=250}=1583$, for the $L$-function associated to elliptic curve $y^2-x^3-1217 x^2-96135 x=0$, of conductor $C=421666952460$ \cite{PenneyPomerance}. The closest 5 zeros to the central point are clearly visible.}
\label{fig4}
\end{figure}

\section{Preliminary facts}
\label{sec2prelim}

An elliptic curve $E$ over $\mathbb{Q}$ is a non-singular cubic curve given by an equation
\be
y^2 = ax^3 + bx^2 + cx +d,
\ee
with $a,b,c,d \in \mathbb{Q}$,  with $a \neq 0$.  Let $p$ be a prime,  and let $E_p$ be the curve over the finite field $\mathbb{F}_p$ obtained by reducing the coefficients of $E$ modulo $p$.  We say that $E$ has good reduction at $p$ if $E_p$ is nonsingular,  and has bad reduction at $p$ otherwise.  The types of bad reduction at $p$ can be further classified into: additive (if $E_p$ has a cusp at its singular point),  split multiplicative (if the singular point of $E_p$ is a node having two tangent lines with slopes in $\mathbb{F}_p$,  rather than in an extension),  and non-split multiplicative reduction.  The conductor of $E$ is defined to be
\be
C = \prod_{p} p^{e_p},
\ee
where the exponent $e_p$ is zero if $E$ has good reduction at $p$,  one if $E$ has multiplicative reduction at $p$,  and two if $E$ has additive reduction and $p \neq 2,3$ (if $p=2$ or $p=3$,  the additive reduction case has a more complicated expression).

We define the $L$-series of an elliptic curve $E$ with conductor $C$ as the Euler product 
\be
L(E,s) = \prod_{p \nmid C} \frac{1}{1- a_p p^{-(s+1/2)} + p^{-2s}} \prod_{p \mid C} \frac{1}{1+a_p p^{-(s+1/2)}}
\ee
which converges when $\Re(s)>1$.  If $E$ has good reduction at $p$, $a_p = p+1 - n_p$,  where $n_p$ is the number of points of $E_p$.   Otherwise,  $a_p$ is zero in the case of bad additive reduction at $p$,  $-1$ in the case of split multiplicative reduction and $+1$ otherwise.

The function $L(E,s)$ admits an analytic continuation to the entire complex plane and a functional equation that relates the values at $s$ with those at $1-s$.

The critical line $\text{Re}(s) = 1/2$ of symmetry of the functional equation encodes important arithmetic information on the elliptic curve. The famous Birch and Swinnerton-Dyer conjecture posits that the analytic rank of $E$ which is defined to be the order of vanishing of $L(E, 1/2)$, equals the rank of the finitely generated abelian group of rational points of $E$.

For the classical theory on the Riemann zeta function see \cite{Titchmarsh} or \cite{Davenport}. For more on the theory of $L$-functions associated to elliptic curves the reader is referred to \cite{KowalskiIwaniec}.
 
Our notations are as follows:

\begin{center}
\begin{tabular}{l|l}
	$s$ & the complex argument\\
	$C$ & the conductor \\
	$w=(-1)^P$ & the root number \\
	$L_p(E,s)$ & the local factor at $p$ \\
	$g(s)$ & the local factor at $\infty$ \\
	$s_\star$ & complex variable running over the local factor poles \\
	$\mathfrak{S}_N$ & set of poles of all the local factors\\
	$\mathfrak{O}_N$ & set of poles of all the local factors at primes $p$\\
	$r_\star$ & residue at $s_\star$ of the local factor at $p$\\
	$R_{-1/2-l}$ & residue of $g(s)$ at negative half-integer $-1/2-l$ \\
	$p_N$ & the cutoff in the number of primes
\end{tabular}
\end{center}

The local factor at the Archimedean place in analytic convention is
\be
\label{eq21forg}
g(s) = C^\frac{s}{2}\Gamma_\mathbb{C}\lb s + \frac{1}{2}\rb,
\ee
where
\be
\label{eq22forgamma}
\Gamma_\mathbb{C}(z) \coloneqq 2 (2 \pi)^{-z} \Gamma(z),
\ee
with $\Gamma(z)$ the Euler Gamma function. \extratext{In analytic convention the critical strip is $0<\Re(s)<1$ (see the $L$-function and modular forms database \href{https://www.lmfdb.org/knowledge/show/lfunction.normalization}{LMFDB.org} for more details).}
\begin{remark}
In analytic convention, the inverse local factors at places of bad reduction are $1$, $1+p^{-s-1/2}$, or $1-p^{-s-1/2}$, depending on whether the bad reduction is additive, nonsplit multiplicative, or split multiplicative respectively.
\end{remark}

We denote the local factor at $p$ in analytic convention by $L_p(E,s)$, so that in the case of good reduction $L_p^{-1}(E,s)=1 - a_p p^{-(s+1/2)} + p^{-2s}$, and similar expressions hold at the places of bad reduction.

\begin{remark}[Hasse bound] The Hasse bound is
\be
|a_p|<2\sqrt{p}.
\ee
\end{remark}

Coefficient $a_p$ is related to the number of points $n_p$ by the relation
\be
a_p = p+1 -n_p. 
\ee

\extratext{The information in the $L$-function, according to our construction, will depend on the local $L$-factor poles. Lemmas \ref{lemma2} -- \ref{lemmafourpoles} characterize the locations and types of these poles, in terms of reduction type.}

\begin{lemma}
\label{lemma2}
In analytic convention, the poles of a local factor with good reduction are on the imaginary axis, and the poles of a local factor with bad multiplicative reduction are on the imaginary line of real part\ \ $-1/2$.
\end{lemma}

For each prime $p$, $L_p(E,s)$ has simple poles. The inverse of the local factor with good reduction~is
\be
L^{-1}_p(E,s) = 1 - a_p p^{-(s+1/2)} + p^{-2s}.
\ee
If $a_p=0$ the result is immediate. Otherwise, the poles correspond to $L_p^{-1}(E,s)=0$, and are at 
\be
\label{s12is}
s_\star = \frac{2\pi i}{\ln p} k + \frac{1}{\ln p} \ln \lb \frac{a_p\pm i\sqrt{4p-a_p^2}}{2\sqrt{p}} \rb, \quad k\in\mathbb{Z},
\ee
with residue
\be
\label{r12is}
r_\star=\frac{1}{2 \ln p} \lb 1\mp\frac{i a_p}{\sqrt{4 p-a_p^2}} \rb.
\ee
Note that precisely for $|a_p|<2\sqrt{p}$ (from the Hasse bound) the complex number inside the logarithm has unit norm, because
\be
\label{Eq28}
\lb \frac{a_p}{2\sqrt p} \rb^2 + \lb \frac{ \sqrt{4p-a_p^2} }{2\sqrt p} \rb^2 = 1.
\ee
Thus the logarithm is purely imaginary, which proves the lemma in the case of good reduction. In the case of bad multiplicative reduction, the poles are the solutions of equation 
\be
\label{eq19Lan}
L^{-1}_p(E,s) = 1 \pm p^{-s-\frac{1}{2}} = 0,
\ee
so they will lie on the imaginary line of real part $-1/2$.

\begin{remark}
Due to the strict inequality in the Hasse bound, the poles of local factors with good reduction cannot be at $s=0$, and similarly the poles of local factors with nonsplit bad multiplicative reduction cannot be on the real axis. However, the split multiplicative local factors all have a pole on the real axis, at $s=-1/2$, which is also a pole for the local gamma factor at the Archimedean place.
\end{remark}

\begin{lemma}
\label{lemma2again}
All the poles of the local factors with bad multiplicative reduction are distinct, except for the possible poles at $s=-1/2$.
\end{lemma}
\begin{proof}
The locations of the poles are given by Eq. \eqref{eq19Lan}. Because the ratios of logarithms of prime numbers are irrational, the poles cannot coincide, except at $-1/2$.
\end{proof}

\begin{lemma}
\label{lemmafourpoles}
Two distinct local factors with good reduction can have at most four common poles.
\end{lemma}
\begin{proof}
Suppose the local factors with good reduction at $p$ and $q$, $p\neq q$, have a common pole at $s_\star$,~then
\be
s_\star = \frac{2\pi i}{\ln p} k + \frac{1}{\ln p} \ln \lb \frac{a_p\pm_p i\sqrt{4p-a_p^2}}{2\sqrt{p}} \rb = \frac{2\pi i}{\ln q} l + \frac{1}{\ln q} \ln \lb \frac{a_q\pm_q i\sqrt{4q-a_q^2}}{2\sqrt{q}} \rb,
\ee
for $k,l\in\mathbb{Z}$ and one choice of signs $\pm_p$, $\pm_q$. There are four possible choices of signs total and for each choice there can be at most one pair of integers $k,l$ (because the ratio $\ln q / \ln p$ is irrational), thus the local factors can share at most four poles.
\end{proof}

For $\sigma,t\in \mathbb{R}$, $-\sigma \notin \mathbb{N}$, the gamma function obeys the identity
\be
|\Gamma\lb \sigma + it \rb| = |\Gamma(\sigma)|\prod_{k=0}^\infty \lb 1 + \frac{t^2}{(\sigma+k)^2} \rb^{-\frac{1}{2}},
\ee
which can be used to prove the following standard result \extratext{that we will need later}.

\begin{lemma}
For real $\sigma\geq 1$ and real $t$ we have
\be
\label{Stirling}
\left| \frac{\Gamma\lb \sigma+ it \rb}{\Gamma\lb \sigma \rb} \right| \ll e^{-\frac{3}{8} \min \lb |t|, \frac{t^2}{\sigma} \rb}.
\ee
\end{lemma}

\section{A family of approximations for Hasse-Weil $L$-functions}
\label{sec3Matiya}

As we will show, the generalization of Matiyasevich's prescription \cite{Ma1} for the case of $L$-functions associated to elliptic curves constructs a function $\Lambda_N$ that approximates the $L$-function $\Lambda$ in the limit $N\to\infty$. Here $N$ is the number of primes included in the approximation, and when $N$ is large $\Lambda_N(E,s)$ approximates $\Lambda(E,s)$ with exponential precision in $p_{N+1}$,
\ba
\Lambda(E,s_0) - \Lambda_N(E,s_0) &=&\frac{ a_{p_{N+1}} C^\frac{1}{4}}{\pi p_{N+1}} e^{-\frac{2 \pi  p_{N+1}}{\sqrt{C}}}  \lsb 1 + (-1)^P + \frac{(2s_0-1)C^\frac{1}{2}}{4\pi p_{N+1}} \lb 1 - (-1)^P \rb + \mathcal{O}(p_{N+1}^{-2}) \rsb \nn \\
& & + \mathcal{O}\lb \frac{1}{\lb p_{N+2}\rb^{1-(-1)^P/2}} e^{-\frac{2\pi p_{N+2}}{\sqrt{C}}} \rb
\ea
(see Theorem \ref{thm2approx} for the details). $\Lambda_N$ is a sum of two terms,
\be
\label{eq23Matiyasevich}
\Lambda_{N}(E,s) \coloneqq \Lambda^\mrm{ingoing}_{N}(E,s) + (-1)^P \Lambda^\mrm{ingoing}_{N}(E,1-s),
\ee
where the root number $(-1)^P$ is the sign in the $\Lambda(E,s)$ functional equation, which reads
\be
\Lambda(E,s) = (-1)^P \Lambda(E,1-s).
\ee

In order to construct the function $\Lambda^\mrm{ingoing}_{N}(E,s)$ we must first construct a finite Euler product $\Lambda_N^\mrm{Euler}(E,s)$, that is
\be
\label{eq114}
\Lambda_N^\mrm{Euler}(E,s) \coloneqq g(s) \prod_{p=2}^{p_N} L_p(E,s),
\ee
where the product is understood to run over primes. The right-hand side of Eq. \eqref{eq114} is a meromorphic function with an infinite number of poles in the complex plane, at $s=s_\star$ with $L^{-1}_p(s_\star)=0$ for all primes $2\leq p \leq p_N$, and at $s=-l-1/2$, $l\in\mathbb{N}$, corresponding to the poles of $g(s)$. Note that if the curve has split multiplicative reduction then $s=-1/2$ is a pole both for $g(s)$ and the split multiplicative local factors, and so $\Lambda_N^\mrm{Euler}(E,s)$ will have a higher order pole at $s=-1/2$.

As explained in Lemmas \ref{lemma2}, \ref{lemma2again}, any pole of a bad reduction local factor cannot coincide with any other poles, except possibly at $s=-1/2$. In the rest of the paper we will remain agnostic whether poles of different good reduction local factors can coincide. For each good reduction local factor pole $s_\star$ we will assume an order $k_\star\geq 1 $ of the pole in $\Lambda_N^\mrm{Euler}(E,s)$. Note that, because we are considering a finite number of primes $p\leq p_N$, $k_\star$ must be finite, from Lemma \ref{lemmafourpoles}.

The next step is to construct the principal part $\Lambda^\mrm{pp}_N(E,s)$ of $\Lambda_N^\mrm{Euler}(E,s)$. Consider the infinite set $\mathfrak{S}_N$ of all the local factor poles,
\be
\mathfrak{S}_N \coloneqq \lcb s_\star | L_{p\leq p_N}^{-1}(E,s_\star) = 0 \rcb \cup \lcb s_\star | g^{-1}(s_\star) =0 \rcb.
\ee
The Laurent series of $\Lambda_N^\mrm{Euler}(E,s)$ around each $s_\star\in\mathfrak{S}_N$,
\be
\label{eq116Laurent}
\Lambda_N^\mrm{Euler}(E,s) = \sum_{k=-k_\star}^\infty \rho^{(k)}_\star\lb s-s_\star\rb^k,
\ee 
has a principal part
\be
\Lambda_N^{\mrm{pp},s_\star}(E,s) \coloneqq \sum_{k=-k_\star}^{-1} \rho^{(k)}_\star\lb s-s_\star\rb^k.
\ee
The principal part $\Lambda^\mrm{pp}_N(E,s)$ of $\Lambda_N^\mrm{Euler}(E,s)$ is defined as the sum
\be
\label{eq118ppdefn}
\Lambda^\mrm{pp}_N(E,s) \coloneqq \sum_{s_\star \in \mathfrak{S}_N} \Lambda_N^{\mrm{pp},s_\star}(E,s),
\ee
and $\Lambda^\mrm{ingoing}_{N}(E,s)$ is defined as
\be
\label{eq119Lambdain}
\Lambda^\mrm{ingoing}_{N}(E,s) \coloneqq \Lambda^\mrm{Euler}_{N}(E,s) - \Lambda^\mrm{pp}_{N}(E,s).
\ee

The principal part $\Lambda_N^\mrm{pp}(E,s)$ in Eq. \eqref{eq118ppdefn} is well-defined, meaning that the sum over $s_\star \in \mathfrak{S}_N$ converges, as we explain in Remarks \ref{remconv1}, \ref{remconv2}.

\begin{remark}
\label{remconv1}
Let
\be
\mathfrak{O}_N \coloneqq \mathfrak{S}_N-\{-l-1/2\,|\,l\in\mathbb{N}\}
\ee
be the set of local factor poles not at the negative half-integers. The sum in Eq. \eqref{eq118ppdefn} over $s_\star\in \mathfrak{O}_N \subset \mathfrak{S}_N$ converges, because the factor of $g(s_\star)$ decays rapidly at large absolute value of $\Im\lb s_\star \rb$. 
\end{remark}

\begin{remark}
\label{remconv2}
Let $R_{-l-1/2}$ be the residue of $g(s)$ at the negative half-integer $-l-1/2$, $l\in\mathbb{N}$. We have that $R_{-1/2}=2/C^{1/4}$, and from the identity $\Gamma(z+1)=z\Gamma(z)$ it follows that
\be
R_{-(l+1)-1/2} = R_{-l-1/2}\frac{-2\pi}{\sqrt{C}(l+1)},
\ee
so that the residues at the negative half-integers decay rapidly in the real negative direction. Thus the sum in Eq. \eqref{eq118ppdefn} over $s_\star\in \{-l-1/2\,|\,l\in\mathbb{N}\} \subset \mathfrak{S}_N$ converges.
\end{remark}

\begin{remark}
$\Lambda^\mrm{ingoing}_{N}(E,s)$ in Eq. \eqref{eq119Lambdain} and $\Lambda_{N}(E,s)$ in Eq. \eqref{eq23Matiyasevich} are entire.
\end{remark}

\begin{example}
To illustrate this prescription on an example, consider an $L$-function arising from an elliptic curve with only bad additive reduction. Furthermore, suppose that all orders of the good place local factor poles are $k_\star=1$. Then the prescription reads
\ba
\label{eq24Matiyasevich}
\Lambda^\mrm{ingoing}_{N}(E,s) &=& g(s) \prod_{p=2}^{p_N} L_p(E,s)- \sum_{s_{\star}\in \mathfrak{O}_N} g(s_\star) \prod_{\substack{p=2\\p\,\nmid\, C\\L^{-1}_p(s_\star)\neq 0}}^{p_N} L_p(E,s_\star) \frac{r_\star}{s-s_\star}\nn\\
&-& \sum_{l=0}^{\infty} \prod_{\substack{p=2\\p\,\nmid\, C}}^{p_N} L_p(E, -l-1/2) \frac{R_{-l-1/2}}{s+l+1/2}, \nn
\ea
with $s_\star$ and $r_\star$ given by Eqs. \eqref{s12is} and \eqref{r12is} respectively.
\end{example}

\section{The effectiveness and error of the approximations}
\label{sec4someproofs}

We now prove that our family approximates the degree $2$ $L$-function $\Lambda(E,s)$. \extratext{Our argument will involve integrating along certain contours in the complex plane that, for an approximation $\Lambda_N(E,s)$, are chosen to avoid in a controlled manner all the local factor poles in $\Lambda_N^\mrm{Euler}(E,s)$. These contours are introduced in Definition \ref{sparsedef}. Lemmas \ref{shortlemma1p5ish}, \ref{lemma1p5bound} characterize these contours and obtain upper bounds for $\Lambda_N^\mrm{pp}(E,s)$ for $s$ on the contours.}

\begin{defn}
\label{sparsedef}
A closed contour $\mathcal{C}$ in the complex plane is sparse w.r.t. $\mathfrak{S}_{N}$ if for $s\in\mathcal{C}$ we have
\be
\label{neweq215}
\min_{p\leq p_N} \min_{s_\star \in \mathfrak{O}_{N}} \left| s - s_\star \right| \gg \frac{1}{p_N}
\ee
and
\be
\label{neweq216}
\min_{l\in\mathbb{N}} |s+l +\frac{1}{2}|\gg 1.
\ee
\end{defn}

\begin{lemma}
\label{shortlemma1p5ish}
There exist arbitrarily large rectangular contours $\mathcal{C}$ (in the sense that the coordinates of the $4$ vertices can be arbitrarily large in the real and imaginary directions) that are sparse w.r.t.~$\mathfrak{S}_{N}$.
\end{lemma}

\begin{proof}
To satisfy Eq. \eqref{neweq216}, it suffices to pick the vertical part of the contour to pass through the middle point between two poles of the local gamma factor. For Eq.~\eqref{neweq215}, note that on any vertical interval of length $1$ on the imaginary axis, from the prime number theorem there are $\mathcal{O}\lb p_N/\ln p_N \rb$ prime numbers $\leq p_N$. For prime $p$ the poles are spaced $\mathcal{O}\lb1/\ln p\rb$ apart, so there will be $\mathcal{O}\lb \ln p \rb$ poles in an interval of length $1$, so in such an interval there will be at most $\mathcal{O}\lb p_N \rb$ poles total. Then, by the pigeonhole principle, it is possible to draw the horizontal part of the contour so that Eq. \eqref{neweq215} holds.
\end{proof}

\begin{lemma}
\label{lemma1p5bound}
With the notations above, uniformly for all $s\in\mathbb{C}$ on a sparse contour wrt. $\mathfrak{S}_N$, we have
\be
\left| \Lambda^\mrm{pp}_{N}(E,s) \right| \ll \frac{1}{1+|s|},
\ee
where the constant implied by the $\ll$ symbol depends on the $L$-function and $N$ only.  
\end{lemma}

\begin{proof}

From Eqs. \eqref{eq21forg}, \eqref{eq22forgamma} we have
\be
g(s) = \frac{2^{\frac{3}{2}-s} \pi^{-s-\frac{1}{2}} C^{s/2}}{1 + 2s} \Gamma\left(s+\frac{3}{2}\right),
\ee
so that from Eq. \eqref{Stirling}, for $s$ of large imaginary part on the imaginary axis, we have
\be
\label{eq219}
|g(s)| \leq C_1 e^{-C_2 \Im\lb s \rb},
\ee
for some constants $C_{1,2}>0$. Consider now the poles at $s_\star$ with $L\leq \Im(s_\star)<L+1$ for nonnegative integer $L$, Eq. \eqref{eq219} implies that
\be
\label{eq128poles}
\max_{L\leq s_{\star} < L+1} \max_{1\leq k \leq k_\star} \left| \rho^{(-k)}_\star\right| \leq  C_{N} e^{-C_0 L},
\ee
where $k_\star$ is the order of the pole at $s_\star$ and $\rho_\star^{(-k)}$ are the coefficients in the Laurent expansion \eqref{eq116Laurent}. This inequality holds due to the factor of $g(s_\star)$ in each $\rho^{(-k)}_\star$. Here the constant $C_0$ depends only on the given $L$-function, and the constant $C_{N}$ depends only on $N$ and the $L$-function. We remark that there exists an $L_0$ depending only on $N$ and the $L$-function, such that for all $L\geq L_0$ all the poles in Eq. \eqref{eq128poles} will be simple, by Lemma \ref{lemmafourpoles} and the fact that the poles in the local factors not at $s=-1/2$ are simple. It follows that, for all $L\geq L_0$, Eq. \eqref{eq128poles} reduces to
\be
\max_{\substack{p_j\leq p_N\\ p_j \text{\, not\, bad\, additive}}}\max_{\substack{L\leq s_{\star} < L+1\\ L_{p_j}(s_\star) = 0 }} \Bigg| g(s_\star) r'_\star \Bigg( \prod_{\substack{p=2\\p\neq p_j\\p\,|\, C}}^{p_N} \frac{1}{1 - a^\mrm{bad}_p p^{-s_\star-1/2}} \Bigg) \Bigg( \prod_{\substack{p=2\\p\neq p_j\\p\,\nmid\, C}}^{p_N} \frac{1}{1 - a_p p^{-(s_\star+1/2)} + p^{-2s_\star}} \Bigg) \Bigg| \leq  C_{p_N} e^{-C_0 L},
\ee
where $a_p^\mrm{bad}\in \{-1,0,1\}$, depending on the type of reduction, and
\be
r'_\star=
\begin{cases}
r_\star \text{\quad if\ } p_j \nmid C, \\
\frac{1}{\ln p_j} \text{\ if\ } p_j \text{\ is\ bad\ multiplicative}
\end{cases}.
\ee

We now consider an $|s|> 2L_0$, and split into four cases,
\ba
L \begin{cases}
\leq L_0 \\
\in \lb L_0\ ,\ceil{\frac{|s|}{2}}\rsb \\
\in \lb \ceil{\frac{|s|}{2}}, \ceil{2|s|} \rb \\
\geq \ceil{2|s|}
\end{cases}.
\ea

For $L\leq L_0$, we have that
\be
\label{eqboundpart0}
\sum_{L=1}^{L_0} \sum_{L\leq s_{\star} < L+1} \sum_{1\leq k \leq k_\star} \frac{\left| \rho^{(-k)}_\star\right|}{|s-s_\star|^k} \leq  \frac{C^{(0)}_{N}}{1+|s|}.
\ee

For $L \in \lb L_0\ ,\ceil{\frac{|s|}{2}}\rsb$, we have that 
\ba
\label{eqboundpart1}
\sum_{L=L_0+1}^{\ceil{\frac{|s|}{2}}}\sum_{\substack{p_j\leq p_N\\ p_j \text{\, not\, bad\, additive}}}\sum_{\substack{L\leq s_{\star} < L+1\\ L_{p_j}(s_\star) = 0 }} \Bigg| g(s_\star) r'_\star \Bigg( \prod_{\substack{p=2\\p\neq p_j\\p\,|\, C}}^{p_N} \frac{1}{1 - a^\mrm{bad}_p p^{-s_\star-1/2}} \Bigg) \times \\ \times \Bigg( \prod_{\substack{p=2\\p\neq p_j\\p\,\nmid\, C}}^{p_N} \frac{1}{1 - a_p p^{-(s_\star+1/2)} + p^{-2s_\star}} \Bigg) \Bigg| \frac{1}{|s-s_\star|} \leq \frac{C^{(1)}_{N}}{1+|s|}. \nn
\ea

For $L\geq \ceil{2|s|}$, we have that
\be
\label{eqboundpart2}
\sum_{L=\ceil{2|s|}}^\infty e^{-C_0L} = \frac{e^{\lb 1 - \ceil{2|s|} \rb C_0}}{e^{C_0}-1} \leq \frac{C^{(2)}_{N}}{1+|s|}.
\ee

In Eqs. \eqref{eqboundpart0} -- \eqref{eqboundpart2}, $C^{(0)}_{N}$ $C^{(1)}_{N}$, $C^{(2)}_{N}$ are constants that depend on $N$ and on the $L$-function.

For $\ceil{|s|/2} < L < \ceil{2|s|}$, we have that $\sum_{L=\ceil{\frac{|s|}{2}}}^{\ceil{2|s|}} e^{-C_0L}$ decays exponentially when $|s|$ is large, and the sparseness condition ensures that there can be no large contribution in $s$ coming from the factors of $|s-s_\star|$ in the denominator. We have thus obtained
\be
\label{eq2p22decayatinfty}
\left| \Lambda^\mrm{pp}_{N}(E,s) \right| \leq \frac{C'_{N}}{1+|s|}
\ee
uniformly for $s$ on the sparse contours w.r.t. $\mathfrak{S}_{N}$, where $C'_{N}$ is a constant that depends on $N$, the $L$-function being considered, as well as on the implicit constants in Eqs. \eqref{neweq215}, \eqref{neweq216}.
\end{proof}

\extratext{We now need to estimate the difference between the $L$-function $\Lambda(E,s)$ and the approximation $\Lambda_N(E,s)$. Theorem \eqref{thisisatheorem3} expresses this difference as an integral of $\Lambda(E,s) - \Lambda^\mrm{Euler}_{N}(E,s)$ on a vertical line to the right of the critical strip. This integral presentation will allow us to obtain asymptotic formulas for the difference $\Lambda(E,s_0)-\Lambda_N(E,s_0)$ at any given point $s_0$. }

\begin{thm}
\label{thisisatheorem3}
For any $s_0\in\mathbb{C}$, elliptic curve $L$-function $\Lambda(E,s)$, and $\sigma>1$, we have
\be
\label{eq414thm}
\Lambda(E,s_0) - \Lambda_N(E,s_0) = \frac{1}{2\pi i} \int_{\Re(s)=\sigma}\lsb \Lambda(E,s) - \Lambda^\mrm{Euler}_{N}(E,s) \rsb\lb  \frac{1}{s-s_0} + (-1)^P \frac{1}{s-1+s_0}\rb,
\ee
where $\Lambda_N(E,s_0)$ and $\Lambda_N^\mrm{Euler}(E,s)$ are defined by Eqs. \eqref{eq23Matiyasevich} and \eqref{eq114} above.
\end{thm}

\begin{proof}

Let's consider a simple closed curve $\mathfrak{C}$, which does not need to be sparse in the sense of Definition \ref{sparsedef}, that encloses points $s_0$ and $1-s_0$ and does not pass through the poles arising from the local factors. Let
\be
\label{eq224poles}
\mathcal{I}\lb N, s_0, \mathfrak{C} \rb \coloneqq \frac{1}{2\pi i} \int_\mathfrak{C} \Lambda^\mrm{pp}_N(E,s) \lb \frac{1}{s-s_0} + (-1)^P \frac{1}{s-1+s_0} \rb ds,
\ee
and let $\mathcal{C}_k$, $k\in\mathbb{N}$, be a sequence of rectangular contours that are sparse w.r.t. $\mathfrak{S}_\mrm{p_N}$ in the sense of Definition \ref{sparsedef}. These contours tend to infinity, meaning that each $\mathcal{C}_k$ is contained in $\mathcal{C}_{k+1}$ and the union of the $\mathcal{C}_k$'s is the entire complex plane. Then we can write
\be
\label{eq224}
\mathcal{I}\lb N, s_0, \mathfrak{C} \rb = \frac{1}{2\pi i} \int_{\mathcal{C}_k} \Lambda^\mrm{pp}_N(E,s) \lb \frac{1}{s-s_0} + (-1)^P \frac{1}{s-1+s_0} \rb ds - \sum_\rho \mrm{Res}\lb \rho \rb,
\ee
where $\rho$ runs over all the local factor poles between $\mathcal{C}_k$ and $\mathfrak{C}$, and $\mrm{Res}\lb \rho \rb$ is the residue of $\Lambda^\mrm{pp}_N(E,s) \lsb 1/(s-s_0) + 1/(s-1+s_0) \rsb$ at $\rho$.

From Eq. \eqref{eq2p22decayatinfty}, when $|s|$ is large we have that
\be
\Lambda_N^\mrm{pp}(E,s) \lb \frac{1}{s-s_0} + (-1)^P \frac{1}{s-1+s_0} \rb \ll_{p_N,s_0} \frac{1}{|s|^2},
\ee
so that the integral in Eq. \eqref{eq224} goes to zero as $k\to\infty$. Then we are left with
\be
\label{eq227poles}
\mathcal{I}\lb N, s_0, \mathfrak{C} \rb = - \sum_\rho \mrm{Res}\lb \rho \rb,
\ee
where the $\rho$ sum runs over all the local factor poles outside curve $\mathfrak{C}$ (and, from this argument, the sum converges).

Now, as in \cite{NastasescuZaharescu}, we fix real numbers $T>1$, $\tau>1$, $\sigma>1$, and consider two counterclockwise rectangular contours $\mathcal{C}$, $\mathcal{C}'$ with vertices $\sigma-iT$, $\sigma+iT$, $-\tau+iT$, $-\tau-iT$, and respectively $1+\tau-iT$, $1+\tau+iT$, $-\tau+iT$, $-\tau-iT$. Furthermore, we choose $T$, $\sigma$, $\tau$ sufficiently large so that $s_0$ lies inside both contours.

We have
\be
\label{eqweneedthis419}
\Lambda_N^\mrm{ingoing}(E,s_0) = \frac{1}{2\pi i} \int_{\mathcal{C}'} \frac{\Lambda_N^\mrm{ingoing}(E,s) }{s-s_0} ds, \quad \Lambda_N^\mrm{ingoing}(E,1-s_0) = \frac{1}{2\pi i} \int_{\mathcal{C}'} \frac{\Lambda_N^\mrm{ingoing}(E,s) }{s-1+s_0} ds,
\ee
and so
\be
\Lambda_N(E,s_0) = \frac{1}{2\pi i} \int_{\mathcal{C}'} \Lambda_N^\mrm{ingoing}(E,s) \lb \frac{1}{s-s_0} + (-1)^P \frac{1}{s-1+s_0} \rb ds.
\ee
Since $\Lambda^\mrm{ingoing}_{N}(E,s)=\Lambda^\mrm{Euler}_{N}(E,s)-\Lambda^\mrm{pp}_{N}(E,s)$, we can split this formally as
\be
\label{eqweneedthis}
\Lambda_N(E,s_0) = \frac{1}{2\pi i} \int_{\mathcal{C}'} \Lambda_N^\mrm{Euler}(E,s) \lb \frac{1}{s-s_0} + (-1)^P \frac{1}{s-1+s_0} \rb ds + E\lb N, s_0, T,\tau \rb,
\ee
where
\be
\label{Eq231poles}
 E\lb N, s_0, T,\tau \rb = - \frac{1}{2\pi i} \int_{\mathcal{C}'} \Lambda_N^\mrm{pp}(E,s) \lb \frac{1}{s-s_0} + (-1)^P \frac{1}{s-1+s_0} \rb ds.
\ee
From Eqs. \eqref{eq224poles}, \eqref{eq227poles}, and \eqref{Eq231poles} we thus have
\be
\label{eq423needthis}
E\lb N, s_0, T,\tau \rb = \sum_\rho \mrm{Res}(\rho),
\ee
where the $\rho$ sum runs over all the local factor poles outside $\mathcal{C'}$, i.e. poles at $-l-1/2$ with $l+1/2>\tau$, and poles at $s_\star$ with $\Im\lb s_\star\rb>T$. We have that (see Remarks \ref{remconv1}, \ref{remconv2}, and the proof of Lemma~\ref{lemma1p5bound})
\be
\label{seeeq424}
E\lb N, s_0, T,\tau \rb\to 0\mrm{\ as\ }T,\tau\to\infty.
\ee

We now write $\Lambda(E,s_0)$ as a contour integral, as
\be
\Lambda(E,s_0) = \frac{1}{2}\lsb \Lambda(E,s_0) + (-1)^P \Lambda(E,1-s_0) \rsb = \frac{1}{4\pi i} \int_{\mathcal{C}'} \Lambda(E,s) \lb \frac{1}{s-s_0} + (-1)^P \frac{1}{s-1+s_0} \rb ds.
\ee
Following \cite{NastasescuZaharescu}, let's define
\be
\label{eq234IpN}
I(N,T,\tau,s_0)\coloneqq \frac{1}{2\pi i} \int_{\mathcal{C}'} \lb \frac{\Lambda(E,s)}{2} - \Lambda_N^\mrm{Euler}(E,s) \rb \lb \frac{1}{s-s_0} +(-1)^P \frac{1}{s-1+s_0} \rb,
\ee
so that by substitution we have
\be
\label{Eq235LambdaIE}
\Lambda(E,s_0) - \Lambda_N(E,s_0) = I(N,T,\tau,s_0) - E(N,T,\tau,s_0).
\ee
Note that
\be
\label{eq236ssymm}
\int_{\mathcal{C}'} \Lambda(E,s) \lb \frac{1}{s-s_0} + (-1)^P \frac{1}{s-1+s_0} \rb ds = \int_{\mathcal{C}'} \lb \frac{\Lambda(E,s)}{s-s_0} + \frac{\Lambda(E,1-s)}{s-1+s_0} \rb ds,
\ee
so that the right-hand side of Eq. \eqref{eq236ssymm} is manifestly invariant under the change of variables $s\to1-s$. Let $\mathcal{C}'_\mathcal{L}$ be the part of contour $\mathcal{C}'$ that is to the left of the critical line, and let $\mathcal{C}'_\mathcal{R}$ be the part to the right. Then, from the $s\to 1-s$ invariance of Eq.~\eqref{eq236ssymm}, the part of the integral in Eq. \eqref{eq236ssymm} on $\mathcal{C}'_L$ equals that on $\mathcal{C}'_R$, and we can write Eq. \eqref{eq234IpN} as
\ba
\qquad I(N,T,\tau,s_0) &=& I_\mathcal{L}(N,T,\tau,s_0) + I_\mathcal{R}(N,T,\sigma,s_0),\\
\label{eqIL239}
\qquad I_\mathcal{L}(N,T,\tau,s_0) &\coloneqq& - \frac{1}{2\pi i} \int_{\mathcal{C}_\mathcal{L}} \Lambda_N^\mrm{Euler}(E,s) \lb \frac{1}{s-s_0} +(-1)^P \frac{1}{s-1+s_0} \rb ds,\\
\label{eqIR240}
\qquad I_\mathcal{R}(N,T,\sigma,s_0) &\coloneqq& \frac{1}{2\pi i} \int_{\mathcal{C}_\mathcal{R}} \lsb \Lambda(E,s) - \Lambda_N^\mrm{Euler}(E,s) \rsb \lb \frac{1}{s-s_0} +(-1)^P \frac{1}{s-1+s_0} \rb ds.
\ea
Note that we have also replaced $\mathcal{C}_L'$, $\mathcal{C}_R'$ by $\mathcal{C}_L$, $\mathcal{C}_R$, since $\mathcal{C}_L'=\mathcal{C}_L$ and there are no poles between $\mathcal{C}$ and $\mathcal{C}'$.

The functions $\Lambda^\mrm{Euler}_N(E,s)$, $\Lambda(E,s)$ contain a gamma factor in $g(s)$,~and so from Stirling's formula the contribution of each of the horizontal segments in Eqs. \eqref{eqIL239}, \eqref{eqIR240} vanishes in the limit $T\to\infty$, so that
\ba
\lim_{T\to\infty} I_\mathcal{L}(N,T,\tau,s_0) &=& - \frac{1}{2\pi i} \int_{\Re\lb s\rb=-\tau} \Lambda_N^\mrm{Euler}(E,s) \lb \frac{1}{s-s_0} +(-1)^P \frac{1}{s-1+s_0} \rb ds,\nn\\
\lim_{T\to\infty} I_\mathcal{R}(N,T,\sigma,s_0) &=& \frac{1}{2\pi i} \int_{\Re\lb s\rb=\sigma} \lsb \Lambda(E,s) - \Lambda_N^\mrm{Euler}(E,s) \rsb \lb \frac{1}{s-s_0} +(-1)^P \frac{1}{s-1+s_0} \rb ds. \nn
\ea 
Furthermore, the integral on the vertical segment in $I_\mathcal{L}$ can be bounded as
\be
\lim_{T\to\infty} \left| I_\mathcal{L}(N,T,\tau,s_0) \right| \ll \lb \prod_{p=2}^{p_N} p^{-\tau_p} \rb \int_{-\infty}^\infty \left|  g(-\tau+it) \lsb \frac{1}{-\tau+it-s_0} + (-1)^P \frac{1}{-\tau+it-1+s_0} \rsb \right| dt ,
\ee
where
\be
\tau_p \coloneqq \begin{cases}
1 \text{\ \ \ if } p \text{ is bad additive}\\
\tau \text{\ \ \ if } p \text{ is bad multiplicative} \\
2\tau \text{ if } p \text{ is good reduction} \\
\end{cases},
\ee
so that $\lim_{T\to\infty} \left| I_\mathcal{L}(N,T,\tau,s_0) \right|$  vanishes in the limit $\tau\to\infty$.

Using Eq. \eqref{Eq235LambdaIE}, and that $E\lb N, s_0, T,\tau \rb\to 0$ vanishes in the limit $T,\tau\to\infty$,  we have thus arrived at
\be
\label{eq240like427}
\Lambda(E,s_0) - \Lambda_N(E,s_0) = \frac{1}{2\pi i} \int_{\Re(s)=\sigma}\lsb \Lambda(E,s) - \Lambda^\mrm{Euler}_{N}(E,s) \rsb\lb  \frac{1}{s-s_0} + (-1)^P \frac{1}{s-1+s_0}\rb ds.
\ee
Eq. \eqref{eq240like427} is an exact relation, and is the analogue of Eq. (4.27) in \cite{NastasescuZaharescu}.
\end{proof}

\extratext{In Theorem \ref{thm2approx} we will push further the approach started with Theorem \ref{thisisatheorem3}, by giving a series for the right-hand side of Eq. \eqref{eq414thm}, with the terms expressed in closed-form. This closed-form expression will allow us to finally obtain the asymptotic formula in Theorem \ref{theoremintro} for the difference $\Lambda(E,s_0) - \Lambda_N(E,s_0)$.}

\begin{thm} 
\label{thm2approx}
For any $N>1$, $s_0\in\mathbb{C}$, elliptic curve $L$-function $\Lambda(E,s)$ with conductor $C$, and approximation $\Lambda_N(E,s)$ defined as in Eq.~\eqref{eq23Matiyasevich}, we have
\ba
\label{eq161}
\Lambda(E,s_0) - \Lambda_N(E,s_0) &=& 2 \sum_{n=p_{N+1}}^\infty c_n \Bigg[ \frac{C^\frac{s_0}{2}}{\left(2\pi n\right)^{s_0+\frac{1}{2}}} \Gamma\left(s_0+\frac{1}{2},\frac{2 \pi n}{\sqrt{C}}\right)\\
&+& (-1)^P \frac{C^\frac{1-s_0}{2}}{\left(2\pi n\right)^{-s_0+\frac{3}{2}}} \Gamma\left(-s_0+\frac{3}{2},\frac{2 \pi n}{\sqrt{C}}\right) \Bigg] \nn
\ea
where
\be
c_n \coloneqq
\begin{cases}
a_n & \text{if\ } p |n \text{\ for some prime } p> p_N\\
0 & \text{otherwise}
\end{cases}.
\ee
\end{thm}

\begin{proof}[Proof of Theorem \ref{theoremintro} and Theorem \ref{thm2approx}]
For $\Re(s)\geq 1$ (in analytic convention) we have the identity
\be
\prod_{p=2}^{p_N} L_p(E,s) = \sum_{n=p_{N+1}}^\infty \frac{c_n}{n^{s+\frac{1}{2}}},
\ee
so that from Eq. \eqref{eq240like427} we have
\be
\label{eq440thisis}
\Lambda(E,s_0) - \Lambda_N(E,s_0) = \frac{1}{2\pi i} \sum_{n=p_{N+1}}^\infty \int_{\Re(s)=\sigma} g(s) \frac{c_n}{n^{s+\frac{1}{2}}} \lb  \frac{1}{s-s_0} + (-1)^P \frac{1}{s-1+s_0}\rb ds.
\ee

Let
\be
J_n(s_0) \coloneqq \int_{\Re(s)=\sigma} g(s) \frac{c_n}{n^{s+\frac{1}{2}}} \lb  \frac{1}{s-s_0} + (-1)^P \frac{1}{s-1+s_0}\rb ds,
\ee
and we define
\be
\label{Kns0is}
K_n(s_0) \coloneqq \int_{\Re(s)=\sigma} \frac{g(s)}{n^{s+\frac{1}{2}}} \frac{1}{s-s_0} ds
\ee
so that
\be
\label{JisKsymmetrized}
J_n(s_0) = c_n\lsb K_n(s_0) + (-1)^P K_n(1-s_0) \rsb.
\ee
We now compute $K_n(s_0)$, using that $\sigma>|s_0|$. For any positive integer $b$, we shift the entire vertical line $\Re(s)=\sigma$ to the left to $\Re(s)=-b$. This is allowed, because when $T\to \infty$ the integral of the argument in Eq. \eqref{Kns0is} on the horizontal lines from $-b\pm iT$ to $\sigma\pm iT$ decays rapidly due to the gamma factor in $g(s)$. This shift picks up a contribution from the poles at $s=s_0$ and at $s=-k-1/2$, $k\in\mathbb{N}$, $k\leq b-1$. Then
\be
\label{Kniswithextraintegral}
K_n(s_0) = 2 \pi i \lsb \mrm{Res}_{s=s_0} \frac{g(s)}{n^{s+\frac{1}{2}}} \frac{1}{s-s_0} + \sum_{k=0}^{b-1} \mrm{Res}_{s=-k-1/2} \frac{g(s)}{n^{s+\frac{1}{2}}} \frac{1}{s-s_0}  \rsb + \int_{\Re(s)=-b} \frac{g(s)}{n^{s+\frac{1}{2}}} \frac{1}{s-s_0} ds.
\ee
From the relation $\Gamma(s+1)=s\Gamma(s)$, for $s=-b+it$, $b\in\mathbb{N}^\times$, we have
\be
\left|\Gamma\lb -b + \frac{1}{2} + it \rb\right| = \frac{\left|\Gamma\lb \frac{1}{2} + it \rb\right|}{\left| b - \frac{1}{2}\right|\left| b - \frac{3}{2}\right| \dots \frac{3}{2}\frac{1}{2} } \leq \frac{2\left|\Gamma\lb \frac{1}{2} + it \rb\right|}{(b-1)!},
\ee
so that
\ba
\left|\int_{\Re(s)=-b} \frac{g(s)}{n^{s+\frac{1}{2}}} \frac{1}{s-s_0} ds\right| &\leq& \frac{4}{(b-1)!} \int^\infty_{-\infty} \left| \frac{C^\frac{-b+it}{2}\Gamma\lb \frac{1}{2} + it \rb}{(2\pi n)^{-b+it+\frac{1}{2}}} \right| \frac{dt}{|-b+it-s_0|} \\
&=& \frac{4 (2\pi n)^{b-\frac{1}{2}} }{(b-1)! C^\frac{b}{2} } \int^\infty_{-\infty} \left| \frac{\Gamma\lb \frac{1}{2} + it \rb}{-b+it-s_0} \right| dt
\ea
vanishes in the limit $b\to\infty$. We have thus obtained
\be
\label{Knis}
K_n(s_0) = 2 \pi i \lsb \mrm{Res}_{s=s_0} \frac{g(s)}{n^{s+\frac{1}{2}}} \frac{1}{s-s_0} + \sum_{k=0}^\infty \mrm{Res}_{s=-k-1/2} \frac{g(s)}{n^{s+\frac{1}{2}}} \frac{1}{s-s_0}  \rsb,
\ee
where
\ba
\label{eq165}
\mrm{Res}_{s=s_0} \frac{g(s)}{n^{s+\frac{1}{2}}} \frac{1}{s-s_0} &=&
   \frac{2C^\frac{s_0}{2}}{\lb 2\pi n\rb^{s_0+\frac{1}{2}}} \Gamma\left(s_0+\frac{1}{2}\right),\\
\label{eqref166}
\mrm{Res}_{s=-k-1/2} \frac{g(s)}{n^{s+\frac{1}{2}}} \frac{1}{s-s_0} &=& -\frac{4(-2\pi n)^k}{k!C^{\frac{1}{4}+\frac{k}{2}}(2s_0+2k+1)}.
\ea
The sum over $k$ can be performed, using Eq. \eqref{eqref166} we have (see \cite{GradRyzhik},~page~941)
\be
\label{eq167}
\sum_{k=0}^\infty \mrm{Res}_{s=-k-1/2} \frac{g(s)}{n^{s+\frac{1}{2}}} \frac{1}{s-s_0} = -\frac{2 C^\frac{s_0}{2}}{\left(2\pi n\right)^{s_0+\frac{1}{2}}} \lsb \Gamma\left(s_0+\frac{1}{2}\right) - \Gamma\left(s_0+\frac{1}{2},\frac{2 \pi n}{\sqrt{C}}\right) \rsb,
\ee
where $\Gamma\lb z,a \rb$ is the incomplete Gamma function,
\be
\Gamma\lb z,a \rb \coloneqq \int_a^\infty t^{z-1} e^{-t} dt.
\ee
Plugging in Eqs. \eqref{eq165} and \eqref{eq167} in Eq. \eqref{Knis}, a cancellation takes place and we obtain
\be
\label{eq169}
K_n(s_0) = \frac{4 \pi i C^\frac{s_0}{2}}{\left(2\pi n\right)^{s_0+\frac{1}{2}}} \Gamma\left(s_0+\frac{1}{2},\frac{2 \pi n}{\sqrt{C}}\right).
\ee
Eq. \eqref{eq169} is an exact result. Thus, for all $n$ with $p|n$ for some $p>p_N$, we have obtained
\be
\label{eq178passing}
J_n(s_0) = 4\pi i a_n \lsb \frac{C^\frac{s_0}{2}}{\left(2\pi n\right)^{s_0+\frac{1}{2}}} \Gamma\left(s_0+\frac{1}{2},\frac{2 \pi n}{\sqrt{C}}\right) + (-1)^P \frac{C^\frac{1-s_0}{2}}{\left(2\pi n\right)^{-s_0+\frac{3}{2}}} \Gamma\left(-s_0+\frac{3}{2},\frac{2 \pi n}{\sqrt{C}}\right) \rsb.
\ee
This completes the proof of Theorem \ref{thm2approx}.

In order to finish the proof of Theorem \ref{theoremintro}, we use the fact that when $n$ is large, $K_n(s_0)$ can be expanded in a series as (\cite{GradRyzhik},~page~942)
\be
K_n(s_0) = \frac{2iC^\frac{1}{4}}{n} e^{-\frac{2 \pi  n}{\sqrt{C}}} \lsb 1 + \frac{(2s_0-1)C^\frac{1}{2}}{4\pi n}+ \mathcal{O}(n^{-2})\rsb.
\ee
Using Eqs. \eqref{eq169} and \eqref{JisKsymmetrized}, we arrive at (for all $n$ with $p|n$ for some $p>p_N$)
\be
J_n(s_0) = \frac{2i a_n C^\frac{1}{4}}{n} e^{-\frac{2 \pi  n}{\sqrt{C}}}  \lsb 1 + (-1)^P + \frac{(2s_0-1)C^\frac{1}{2}}{4\pi n} \lb 1 - (-1)^P \rb + \mathcal{O}(n^{-2}) \rsb.
\ee
In Eq. \eqref{eq440thisis} we now consider separately the term $n=p_{N+1}$, which appears on the left-hand side of Eq. \eqref{eq13intro}, and the sum over terms with $n\geq p_{N+2}$, which we bound in absolute value, using $|a_p|<2\sqrt{p}$. The proof of Theorem \ref{theoremintro} follows.
\end{proof}

\begin{remark}
Using Eq. \eqref{eq178passing}, and that $\Gamma\lb 1,a \rb=e^{-a}$, at the central point we have the exact relation
\be
J_n\lb 1/2\rb = \frac{2i a_n C^\frac{1}{4}}{n} e^{-\frac{2 \pi  n}{\sqrt{C}}}  \lb 1 + (-1)^P \rb.
\ee
When $(-1)^P=-1$, we have $J_n\lb 1/2\rb=0$ and $\Lambda_N(1/2)=0$ for all $n,N$.
\end{remark}

Finally, we prove Theorem \ref{thmextra}.
\begin{proof}[Proof of Theorem \ref{thmextra}]
Following the same steps as in Eqs. \eqref{eqweneedthis419} --  \eqref{eqweneedthis}, we can write
\be
\label{eq456thm22}
\Lambda_N(E,s_0) = \frac{1}{2\pi i} \int_{\mathcal{C}} \Lambda_N^\mrm{Euler}(E,s) \lb \frac{1}{s-s_0} + (-1)^P \frac{1}{s-1+s_0} \rb ds + E\lb N, s_0, T,\tau \rb,
\ee
where $\mathcal{C}$ is a  rectangular contour enclosing points $s_0$, $1-s_0$, with vertices at $\sigma-iT$, $\sigma+iT$, $-\tau + iT$, $-\tau-iT$ for real numbers $T,\tau, \sigma>1$. $E\lb N, s_0, T,\tau \rb$ is the contour integral of $\Lambda^\mrm{pp}_N(E,s)$, and as in Eq. \eqref{eq423needthis}, we can write it as
\be
E\lb N, s_0, T,\tau \rb = \sum_\rho \mrm{Res}(\rho),
\ee
where the $\rho$ sum runs over all the local factor poles outside $\mathcal{C}$.

Now take $T\to\infty$ so that the contribution of the horizontal legs of $\mathcal{C}$ to Eq. \eqref{eq456thm22} vanishes due to the rapid decay of the Gamma function in the imaginary direction. Furthermore, take $\tau\to\infty$ along of a sequence of positive integers, so that the left vertical leg of $\mathcal{C}$ at real part $-\tau$ passes midway between two consecutive poles of $g(s)$, and the contribution of the left vertical leg to Eq.~\eqref{eq456thm22} vanishes. We have (see Eq. \eqref{seeeq424})
\be
E\lb N, s_0, T,\tau \rb\to 0\mrm{\ as\ }T,\tau\to\infty,
\ee
so we obtain
\be
\Lambda_N(E,s_0) = \frac{1}{2\pi i} \int_{\Re(s) = \sigma} \Lambda_N^\mrm{Euler}(E,s) \lb \frac{1}{s-s_0} + (-1)^P \frac{1}{s-1+s_0} \rb ds.
\ee
Taking $N\to \infty$ and invoking Theorem \ref{theoremintro} we obtain
\be
\Lambda(E,s_0) = \frac{1}{2\pi i} \int_{\Re(s)=\sigma} g(s) \sum_{n = 1}^\infty \frac{ a_n }{n^{s+\frac{1}{2}}} \lb \frac{1}{s-s_0} + (-1)^P \frac{1}{s-1+s_0} \rb ds,
\ee
where we used that $\sigma>1$ in order to switch between the Euler product and Dirichlet series presentations of the $L$-function in the integrand.
\end{proof}

\section{\extratext{A sequence of regular polygons}}
\label{sec5proofsforpolygons}

\extratext{In this section we will prove Theorem \ref{thm3ngon}. Our strategy is to analyze the various quantities appearing in Eq. \eqref{eq13intro}. To do this we will employ the Sato-Tate distribution, and existing knowledge of the distribution of gaps between consecutive primes. The strategy is to avoid primes for which $|a_{p_{N+1}}|$ is too small, and also primes for which the gap $p_{N+2}-p_{N+1}$ is too small. Then the two terms on the right-hand side of Eq. \eqref{eq13intro} can be bounded in a convenient manner, and Eq. \eqref{eq13intro} will indeed become an asymptotic formula for the difference $\Lambda(E,s)-\Lambda_N(E,s)$.}

\begin{proof}

We will prove Theorem \ref{thm3ngon} under the assumption that $c_m>0$ (recall that $c_m$ is the first nonzero coefficient in the Taylor series of $\Lambda(E,s)$ around $s=1/2$, see Eq. \eqref{eq15TaylorL}). A similar proof holds when $c_m<0$.

From Theorem \ref{theoremintro} we have
\ba
\label{eq55}
& &\left|\Lambda(E,s) - \Lambda_N(E,s) -\frac{ a_{p_{N+1}} C^\frac{1}{4}}{\pi p_{N+1}} e^{-\frac{2 \pi  p_{N+1}}{\sqrt{C}}}  \lsb 1 + (-1)^P + \frac{(2s-1)C^\frac{1}{2}}{4\pi p_{N+1}} \lb 1 - (-1)^P \rb \rsb \right| \\
&\leq & B_1(\Lambda,\mathcal{R}) \frac{ \left| a_{p_{N+1}}\right| }{p^3_{N+1}} e^{-\frac{2 \pi  p_{N+1}}{\sqrt{C}}} + B_2(\Lambda,\mathcal{R}) \frac{1}{\lb p_{N+2}\rb^{1-(-1)^P/2}} e^{-\frac{2\pi p_{N+2}}{\sqrt{C}}}. \nn
\ea

From the Taylor series for $\Lambda(E,s)$ around $s=1/2$, for all $s\in\mathbb{C}$ with $|s-1/2|\leq 1$ we have
\be
\label{eq57ineq}
\left| \Lambda(E,s) - c_m\lb s - 1/2 \rb^m \right| \leq B_3(\Lambda)|s - 1/2|^{m+2},
\ee
for a constant $B_3(\Lambda)>0$ that depends only on the $L$-function $\Lambda(E,s)$.

Suppose now that $P$ is even. Then from Eq. \eqref{eq55} we have
\ba
\label{eq58ineq}
& &\left|\Lambda(E,s) - \Lambda_N(E,s) -\frac{2a_{p_{N+1}} C^\frac{1}{4}}{\pi p_{N+1}} e^{-\frac{2 \pi  p_{N+1}}{\sqrt{C}}}   \right| \leq B_1(\Lambda,\mathcal{R}) \frac{ \left| a_{p_{N+1}}\right| }{p^3_{N+1}} e^{-\frac{2 \pi  p_{N+1}}{\sqrt{C}}} \\
& & + B_2(\Lambda,\mathcal{R}) \frac{1}{\lb p_{N+2}\rb^{1/2}} e^{-\frac{2\pi p_{N+2}}{\sqrt{C}}}. \nn
\ea
With
\be
G(s) \coloneqq c_m \lb s-\frac{1}{2} \rb^m -\frac{2a_{p_{N+1}} C^\frac{1}{4}}{\pi p_{N+1}} e^{-\frac{2 \pi  p_{N+1}}{\sqrt{C}}}
\ee
we have
\ba
\label{eq510ineq}
\left| G(s) - \Lambda_N(E,s) \right| &\leq& \left| G(s) - \Lambda_N(E,s) + \Lambda(E,s) - c_m \lb s-\frac{1}{2} \rb^m \right| \\
& &+\left| \Lambda(E,s) - c_m \lb s-\frac{1}{2} \rb^m \right|. \nn
\ea
From Eqs. \eqref{eq57ineq}, \eqref{eq58ineq}, \eqref{eq510ineq} we thus obtain
\ba
\label{thisis511}
\quad\quad \left| G(s) - \Lambda_N(E,s) \right| &\leq& B_1(\Lambda,\mathcal{R}) \frac{ \left| a_{p_{N+1}}\right| }{p^3_{N+1}} e^{-\frac{2 \pi  p_{N+1}}{\sqrt{C}}} + B_2(\Lambda,\mathcal{R}) \frac{1}{\lb p_{N+2}\rb^{1/2}} e^{-\frac{2\pi p_{N+2}}{\sqrt{C}}} \\
& &+ B_3(\Lambda)|s - 1/2|^{m+2}.\nn
\ea
Let $w_j$, $j\in\{1,\dots m\}$, be the zeros of $G(s)$, which are at
\be
w_j= \frac{1}{2} + \left| \frac{ 2 a_{p_{N+1}} C^\frac{1}{4}}{\pi p_{N+1}c_m} e^{-\frac{2 \pi  p_{N+1}}{\sqrt{C}}} \right|^\frac{1}{m} \lb \sgn a_{p_{n+1}} \rb^\frac{1}{m} e^\frac{2\pi i j}{m}.
\ee
Fix a small $\delta>0$ and consider the circles $C_{j,\Delta}$ of radius 
\be
\Delta=\delta \left| \frac{ 2 a_{p_{N+1}} C^\frac{1}{4}}{\pi p_{N+1}c_m} e^{-\frac{2 \pi  p_{N+1}}{\sqrt{C}}} \right|^\frac{1}{m}
\ee
centered at $w_j$. For $s\in C_{j,\Delta}$ we have
\ba
\label{eq514G}
|G(s)| &=& |c_m| \left| s - w_1 \right| \dots\left| s- w_m \right| \\
&=& |c_m| \Delta \prod_{\substack{l=1\\l\neq j}}^m |s - w_l|.\nn
\ea
Each factor on the right-hand side of Eq. \eqref{eq514G} satisfies
\ba
\label{eq515G}
|s - w_l| &\geq& |w_j - w_l| - |w_j - s| \\
&=& \left| \frac{ 2 a_{p_{N+1}} C^\frac{1}{4}}{\pi p_{N+1}c_m} e^{-\frac{2 \pi  p_{N+1}}{\sqrt{C}}} \right|^{\frac{1}{m}} \lb \left| e^{\frac{2\pi i j}{m}} - e^{\frac{2\pi i l}{m}} \right| - \delta \rb, \nn
\ea
so that from Eqs. \eqref{eq514G}, \eqref{eq515G} it follows that for $s\in C_{j,\Delta}$
\ba
\label{eq516forrouche}
|G(s)| &\geq& \delta \left| \frac{ 2 a_{p_{N+1}} C^\frac{1}{4}}{\pi p_{N+1}} e^{-\frac{2 \pi  p_{N+1}}{\sqrt{C}}} \right| \prod_{\substack{l=1\\l\neq j}}^m \lb \left| e^{\frac{2\pi i j}{m}} - e^{\frac{2\pi i l}{m}} \right| - \delta \rb\\
&\geq& \delta \left| \frac{ 2 a_{p_{N+1}} C^\frac{1}{4}}{\pi p_{N+1}} e^{-\frac{2 \pi  p_{N+1}}{\sqrt{C}}} \right| E_m,\nn
\ea
where $E_m>0$ depends only on $m$ (assuming that $\delta$ is small enough as a function of $m$).

We need the following three inequalities,
\ba
\label{517start}
B_1(\Lambda,\mathcal{R}) \frac{ \left| a_{p_{N+1}}\right| }{p^3_{N+1}} e^{-\frac{2 \pi  p_{N+1}}{\sqrt{C}}} \leq  \frac{\delta  E_m}{4} \left| \frac{ 2 a_{p_{N+1}} C^\frac{1}{4}}{\pi p_{N+1}} e^{-\frac{2 \pi  p_{N+1}}{\sqrt{C}}} \right|, \\
B_2(\Lambda,\mathcal{R}) \frac{1}{\lb p_{N+2}\rb^{1/2}} e^{-\frac{2\pi p_{N+2}}{\sqrt{C}}} \leq  \frac{\delta  E_m}{4} \left| \frac{ 2 a_{p_{N+1}} C^\frac{1}{4}}{\pi p_{N+1}} e^{-\frac{2 \pi  p_{N+1}}{\sqrt{C}}} \right|, \\
\label{eq5p19}
B_3(\Lambda)\left|s - \frac{1}{2}\right|^{m+2} \leq  \frac{\delta  E_m}{4} \left| \frac{ 2 a_{p_{N+1}} C^\frac{1}{4}}{\pi p_{N+1}} e^{-\frac{2 \pi  p_{N+1}}{\sqrt{C}}} \right|,
\ea
for $s$ on each circle $C_{j,\Delta}$. The first inequality reduces to
\be
\frac{ B_1(\Lambda,\mathcal{R})}{p^2_{N+1}} \leq  \frac{\delta  E_mC^\frac{1}{4}}{2\pi} ,
\ee
which holds true for all $p_{N+1}$ large enough in terms of $\delta$. The second inequality reduces to 
\be
\label{eq521holds}
B_2(\Lambda,\mathcal{R}) \leq  \frac{\delta E_m C^\frac{1}{4} }{2\pi} \frac{ \left| a_{p_{N+1}} \right|  p_{N+2}^{1/2} }{p_{N+1}} e^{\frac{2 \pi \lb  p_{N+2} - p_{N+1} \rb}{\sqrt{C}}}.
\ee
Let $\mathfrak{B}^+_{\Lambda,\delta}$ be the set of primes $p_N$ such that
\be
\label{eq522ap}
a_{p_{N+1}} \geq \delta \sqrt{p_{N+1}},
\ee
and
\be
\label{B2eq523}
B_2(\Lambda,\mathcal{R}) \leq  \frac{\delta^2 E_m C^\frac{1}{4} }{2\pi} e^{\frac{2 \pi \lb  p_{N+2} - p_{N+1} \rb}{\sqrt{C}}}.
\ee
For all $p_N\in \mathfrak{B}^+_{\Lambda,\delta}$, Eq. \eqref{eq521holds} holds. From the Sato-Tate distribution (see \cite{satotate1,satotate3,satotate2}) we know that the set of $p_N$'s that verify Eq.~\eqref{eq522ap} has density
\ba
\label{522lim}
\lim_{M\to\infty} \frac{ \# \{ p_N \leq M: p_N \text{\ satisfies\ Eq.~\eqref{eq522ap}} \} }{\#\{p\leq M\}} &=& \frac{2}{\pi} \int_{0}^{\arccos\frac{\delta}{2}} \sin^2\theta d\theta \nn \\
&=& \frac{-\delta\sqrt{4-\delta^2}+4\arccos\frac{\delta}{2}}{4\pi}\\
&=& \frac{1}{2}-\frac{\delta}{\pi} + O\lb\delta^2 \rb\nn
\ea 
in the full set of primes. Similarly, one defines $\mathfrak{B}^-_{\Lambda,\delta}$ as the set of primes $p_N$ which satisfy Eq.~\eqref{B2eq523}~and
\be
a_{p_{N+1}} \leq -\delta \sqrt{p_{N+1}}.
\ee
The density of $\mathfrak{B}^-_{\Lambda,\delta}$ in the full set of primes is also $1/2-\delta/\pi+O(\delta^2)$.

It is well-known (see for example \cite{sievebook2,sievebook1,cojocaru}) that for any positive even number $d$,
\be
\label{eq523bounds}
\#\{ p_N \leq M : p_{N+2}-p_{N+1}=d \} = O_d\lb \frac{M}{\ln^2 M} \rb.
\ee
Adding bounds \eqref{eq523bounds} for all $d$ up to 
\be
 \frac{\sqrt{C}}{2\pi} \ln \lb \frac{2\pi B_2(\Lambda,\mathcal{R})}{\delta^2 E_m C^{\frac{1}{4}}} \rb
\ee
we obtain that
\be
\label{525lim}
\frac{ \# \{ p_N \leq M: p_N \text{\ satisfies\ Eq. \eqref{B2eq523}} \} }{\#\{p\leq M\}} = 1 - O_{\Lambda,\mathcal{R},\delta}\lb \frac{1}{\ln M} \rb.
\ee
From Eqs. \eqref{522lim} and \eqref{525lim} it follows that
\be
\label{thisis527}
\lim_{M\to\infty} \frac{ \# \{ p_N \leq M: p_N \in \mathfrak{B}^\pm_{\Lambda,\delta} \} }{\#\{p\leq M\}} = \frac{1}{2} - O\lb \delta \rb,
\ee
i.e. both $\mathfrak{B}^+_{\Lambda,\delta}$, $\mathfrak{B}^-_{\Lambda,\delta}$ have density $1/2-O\lb\delta\rb$.

For the third inequality, pick a point $s$ in $C_{j,\Delta}$, and note that
\ba
\left| s - \frac{1}{2} \right| &\leq& \left| s - w_j \right| + \left| w_j - \frac{1}{2} \right| \\
&=& \Delta + \left| \frac{ 2 a_{p_{N+1}} C^\frac{1}{4}}{\pi p_{N+1}c_m} e^{-\frac{2 \pi  p_{N+1}}{\sqrt{C}}} \right|^\frac{1}{m} \nn \\
&=&\lb 1 + \delta \rb \left| \frac{ 2 a_{p_{N+1}} C^\frac{1}{4}}{\pi p_{N+1}c_m} e^{-\frac{2 \pi  p_{N+1}}{\sqrt{C}}} \right|^\frac{1}{m}. \nn
\ea
Therefore,
\be
\label{5p23}
B_3(\Lambda)\left|s - \frac{1}{2}\right|^{m+2} \leq B_3(\Lambda) \lb 1 + \delta \rb^{m+2} \left| \frac{ 2 a_{p_{N+1}} C^\frac{1}{4}}{\pi p_{N+1}c_m} e^{-\frac{2 \pi  p_{N+1}}{\sqrt{C}}} \right|^\frac{m+2}{m}.
\ee
In order to prove Eq. \eqref{eq5p19} it is enough to see that the right-side of Eq. \eqref{5p23} is smaller than the right side of Eq. \eqref{eq5p19}. This is equivalent to
\be
B_3(\Lambda) \lb 1 + \delta \rb^{m+2} \left| \frac{ 2 a_{p_{N+1}} C^\frac{1}{4}}{\pi p_{N+1}c_m} e^{-\frac{2 \pi  p_{N+1}}{\sqrt{C}}} \right|^\frac{2}{m} \leq \frac{\delta  E_m}{4}. 
\ee
This holds true for all $p_{N+1}$ large enough in terms of $\delta$.

Adding Eqs. \eqref{517start} -- \eqref{eq5p19} and using Eq. \eqref{thisis511}, we see that for all $p_N$ that are large enough in terms of $\delta$ and also belong to $\mathfrak{B}^+_{\Lambda,\delta}$ we have that
\be
\label{eq525forrouche}
\left| G(s) - \Lambda_N(E,s) \right| \leq \frac{3\delta E_m}{4} \left| \frac{ 2 a_{p_{N+1}} C^\frac{1}{4}}{\pi p_{N+1}} e^{-\frac{2 \pi  p_{N+1}}{\sqrt{C}}} \right|.
\ee
Combining Eqs. \eqref{eq516forrouche} and \eqref{eq525forrouche}, for $p_N$ large enough and in $\mathfrak{B}^+_{\Lambda,\delta}$, and for $s$ on any of the circles $C_{j,\Delta}$ we have
\be
\left| G(s) - \Lambda_N(E,s) \right| < \left| G(s) \right|.
\ee
For $\delta$ small enough the circles $C_{j,\Delta}$ do not intersect. Then, since $G(s)$ has exactly one root inside each $C_{j,\Delta}$, from Rouch\'e's theorem it follows that $\Lambda_N(E,s)$ has exactly one root inside each of the circles~$C_{j,\Delta}$. Thus, the set $A_{\Lambda,N}$ has exactly one element in each of the following disks
\be
\label{thisis533}
D_{j,\delta} \coloneqq \{ z\in\mathbb{C} : \Bigg|z- \Bigg| \frac{2C^\frac{1}{4}}{\pi c_m}\Bigg|^\frac{1}{m} e^{\frac{2\pi i j}{m}} \Bigg| < \Bigg| \frac{2C^\frac{1}{4}}{\pi c_m} \Bigg|^\frac{1}{m}  \delta \}, \quad j=1,\dots,m,
\ee
centered at the points of $\mathcal{S}_m^\mrm{even,+}$.

Taking into account relations \eqref{thisis527} and \eqref{thisis533} we conclude: For each fixed small $\delta>0$ we constructed a set $\mathfrak{B}^+_{\Lambda,\delta}$ which has density $1/2-O\lb\delta\rb$ and has the additional property that for $p_N\to\infty$ with $p_N$ in $\mathfrak{B}^+_{\Lambda,\delta}$, all the limit points of the sequence of sets $A_{\Lambda,N}$ are at distance $O_{m,\Lambda}(\delta)$ from our fixed set $\mathcal{S}_m^\mrm{even,+}$. Lastly, one now takes $\delta$ tending to zero slowly, to obtain a set $\mathfrak{B}^+_{\Lambda}$ having the properties from the statement of the theorem. This concludes the proof in the (even,$+$) case. The proof in the (even,$-$) case is similar, working with $\mathfrak{B}^-_{\Lambda}$ instead of $\mathfrak{B}^+_{\Lambda}$. The proofs in the (odd,$+$), (odd,$-$) cases are similar, with the additional observation that by construction the function $\Lambda_N(E,s)$ has a zero at the central point, in addition to having zeros in each of the above small disks. This completes the proof of the theorem.

\textbf{Acknowledgments. } B. S. would like to acknowledge the Northwestern University Amplitudes and Insights group, Department of Physics and Astronomy, and Weinberg College for support. The work of B. S. was supported in part by the Department of Energy under Award Number DE-SC0021485.

\end{proof}

\end{document}